\documentclass[11pt,leqno,twoside]{amsart}
\usepackage[T1]{fontenc}
\usepackage[utf8]{inputenc}
\usepackage[babel=true]{csquotes}  
\usepackage{amstext,amsmath,amscd, bezier,indentfirst,amsthm,amsgen,enumerate}
\usepackage[left=2.5cm, right=2.5cm, top=2.5cm,bottom=2.5cm]{geometry} 
\usepackage{todonotes}
\usepackage{subcaption}
\usepackage[all,knot,arc,import,poly]{xy}
\usepackage{amsfonts,color}
\usepackage{amssymb}
\usepackage{latexsym}
\usepackage{epsfig}
\usepackage{graphicx}
\usepackage{srcltx}
\usepackage{enumitem}
\usepackage{accents}
\usepackage{imakeidx}
\usepackage{rotating}
\usepackage{verbatim}
\usepackage{fancyhdr}
\fancyheadoffset[RE,LO]{-0.01\textwidth}
\makeindex
\usepackage{tikz}
\usetikzlibrary{patterns}
\usetikzlibrary{matrix,arrows,decorations.pathmorphing}
\usepackage{tikz-cd}
\tikzset{commutative diagrams/.cd}
\usepackage{framed,lipsum}
\usepackage{float}
\usepackage{pgfplots}
\usetikzlibrary{patterns}
\usetikzlibrary{intersections}
\usetikzlibrary{positioning,shadows,arrows}
\tikzstyle{every node}=[anchor=west, minimum height=3em]
\usetikzlibrary{shapes,snakes}
\definecolor{forestgreen}{rgb}{0.00, 0.39, 0.00} 
\definecolor{blueblue}{rgb}{0.40, 0.00, 1.00} 
\definecolor{sienna}{rgb}{0.33, 0.08, 0.11}

\newcommand{\defi}{\textbf}
\theoremstyle{plain}
\newtheorem{theorem}{Theorem}[section] 

\newtheorem*{theorem*}{Theorem}
\newtheorem*{theoremfr*}{Théorème}
\newtheorem*{hypothesisfr*}{Hypothèse}
\newtheorem*{hypothesis*}{Hypothesis}
\newtheorem{lemma}[theorem]{Lemma}
\newtheorem{proposition}[theorem]{Proposition}
\newtheorem{corollary}[theorem]{Corollary}
\newtheorem{notation}[theorem]{Notation}

\theoremstyle{definition}
\newtheorem{definition}[theorem]{Definition}
\newtheorem*{definition*}{Definition}
\theoremstyle{remark}
\newtheorem{remark}[theorem]{\sc Remark}
\newtheorem*{question*}{\sc Question}
\newtheorem*{remark*}{\sc Remark}
\newtheorem*{remarkfr*}{\sc Remarque}
\newtheorem*{examplefr*}{\sc Exemple}
\newtheorem*{example*}{\sc Example}
\newtheorem{question}[theorem]{\sc Question}
\newtheorem{example}[theorem]{\sc Example}

\def\bR{{\mathbb R}}

\def\min{{\rm min\ }}

\def\const.{{\rm const.}}

\usepackage[style=alphabetic,natbib=true,backref=true,backend=bibtex]{biblatex}
\usepackage{mathscinet}
\usepackage[pdftex, pdfusetitle, plainpages=false,  bookmarks, bookmarksnumbered,colorlinks, linkcolor=blue, citecolor=red,filecolor=black, urlcolor=black]{hyperref}
\usepackage{pgfplots}
\pgfplotsset{width=7cm,compat=1.8}
\usepackage[tight]{minitoc} 
\usepackage{caption}
\makeatletter
\renewcommand*{\numberline}[1]{\hb@xt@1em{#1\hfil}} 
\makeatother

\begin{document}
\title{Constructing Separable Arnold Snakes of Morse Polynomials}
\date{\today}
\author{Miruna-\c Stefana Sorea}
\address{Max-Planck-Institut für Mathematik in den Naturwissenschaften, Leipzig, Germany}
\email{miruna.sorea@mis.mpg.de}
\maketitle
\begin{abstract}
We give a new and constructive proof of the existence of a special class of univariate polynomials whose graphs have preassigned shapes. By definition, all the critical points of a Morse polynomial function are real and distinct and all its critical values are distinct. Thus we can associate to it an alternating permutation: the so-called \emph{Arnold snake}, given by the relative positions of its critical values. We realise any separable alternating permutation as the Arnold snake of a Morse polynomial. 
\end{abstract}
\renewcommand{\contentsname}{}
\section*{Introduction}
Let us consider a Morse polynomial function $P:\bR\rightarrow\bR$ in one variable. By definition, all the critical points of $P$ are real and distinct and all its critical values are distinct. Thus we can associate to $P$ an alternating sequence of real numbers and an alternating permutation called \emph{Arnold snake}. Both are given by its critical values. 

\begin{example*}\label{ExSnakeOfAPolynomial}
Consider the polynomial $P:\mathbb{R}\rightarrow\mathbb{R}$, $P(x):=\int_{0}^x (5t-5)(t-2)(t-4)(t-7)\mathrm{d}t,$ which has four real distinct critical points. The alternating sequence of its critical values is $[20.7,18.4,28.8,-44.1]$. The Arnold snake associated to $P$ is the following alternating permutation 
$\sigma:=\begin{pmatrix}
    1 & 2 & 3 & 4\\
    3 & 2 & 4 & 1
  \end{pmatrix}
  $ (see the figure below). An alternating permutation is simply called a \emph{snake}.

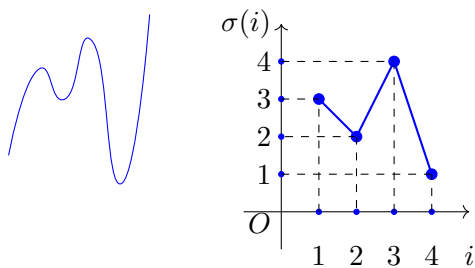
\begin{figure}[H]
\begin{center}
\begin{tikzpicture}[scale=0.5]
\draw[->] (-1,0) -- (5,0) node[below] {$i$};
      \draw[->] (0,-1) -- (0,5) node[left] {$\sigma(i)$};
\tikzstyle{vertex} = [draw,blue]
\tikzstyle{edge} = [draw,dashed]

\begin{scope}[xshift=-8cm,scale=0.75]
    \draw [blue] plot [smooth, tension=1] coordinates { (1,2) (2,5) (3,4) (4,6) (5,1) (6,7)};
\end{scope}
\begin{scope}
\coordinate (A) at (1,3);
\coordinate (B) at (2,2);
\coordinate (C) at (3,4);
\coordinate (M) at (4,1);
\coordinate (O) at (0,-0.3);

\coordinate (D) at (1,0);
\coordinate (E) at (2,0);
\coordinate (F) at (3,0);
\coordinate (P) at (4,0);

\coordinate (G) at (0,1);
\coordinate (H) at (0,2);
\coordinate (I) at (0,3);
\coordinate (R) at (0,4);

\node at (O) [left] {$O$};
\node at (D) [below] {$1$};
\node at (E) [below] {$2$};
\node at (F) [below] {$3$};
\node at (P) [below] {$4$};

\node at (G) [left] {$1$};
\node at (H) [left] {$2$};
\node at (I) [left] {$3$};
\node at (R) [left] {$4$};

\fill[vertex] (A) circle(4pt);
\fill[vertex] (B) circle(4pt);
\fill[vertex] (C) circle(4pt);
\fill[vertex] (M) circle(4pt);

\fill[vertex] (D) circle(2pt);
\fill[vertex] (E) circle(2pt);
\fill[vertex] (F) circle(2pt);
\fill[vertex] (P) circle(2pt);

\fill[vertex] (G) circle(2pt);
\fill[vertex] (H) circle(2pt);
\fill[vertex] (I) circle(2pt);
\fill[vertex] (R) circle(2pt);

\draw[edge] (B)--(E);
\draw[edge] (B)--(H);

\draw[edge] (D)--(A);
\draw[edge] (I)--(A);

\draw[edge] (F)--(C);
\draw[edge] (R)--(C);

\draw[edge] (P)--(M);
\draw[edge] (G)--(M);

\tikzstyle{edge} = [draw, blue, thick]
\draw[edge] (B)--(A);
\draw[edge] (B)--(C);
\draw[edge] (C)--(M);

\end{scope}

\end{tikzpicture}
\end{center}
\vspace{-1.5\baselineskip}
\caption{The graph of $P$ and its Arnold snake $\sigma$.
\label{Fig:ExSnakeOfAPolynomial}}
\end{figure}
\end{example*}

Conversely, two questions appear, both regarding the existence of polynomials whose graphs have prescribed shapes:

\begin{question}\label{quest:altPermExistence}
Given an alternating sequence of real numbers, can we find a Morse polynomial that realises this alternating sequence?
\end{question}

\begin{question}\label{quest:snakesExistence}
Given a permutation $\sigma$ that is a snake, can we find a Morse polynomial that realises $\sigma$ as its Arnold snake?
\end{question} 

The partial positive answer to Question \ref{quest:altPermExistence} was given by Davis in \cite{Da}. He proved the existence of polynomials with prescribed critical values. Nevertheless, Davis's approach was not constructive. A similar proof to the one given by Davis can be found in \cite[page 17]{Do}, where Douady gave a proof of existence for polynomials of degree $4$ and a sketch of the proof for the general case. For more references, see \cite[page 853]{My}. The partial affirmative answer to Question \ref{quest:snakesExistence} was given by Arnold in the theorem below. For the proof and for more information on the enumeration of snakes, see \cite{Ar1}, \cite{CN} or \cite[page 59]{La} and \cite{So1}.

\begin{theorem*}\cite[Theorem  29, page 37]{Ar1}\label{Th:ArnoldsSnakes}
The number of snakes is equal to the number of topologically inequivalent Morse polynomials in one variable.
\end{theorem*}

Our main result is the following theorem (see the precise statement in Section \ref{subsect:realising1var}, Theorem \ref{Th:OriginalConstructionOneVariable}). 

\begin{theorem*}
Let $\sigma$ be a separable snake of $n$ elements. For an appropriate explicit choice of the univariate polynomials $a_i(x)\in\mathbb{R}[x]$, define the polynomial $Q_x(y)\in\mathbb{R}[x][y]$,
$$Q_x(y):=\int_{0}^y\prod_{i=1}^n\big (t-a_i(x)\big )\mathrm{d}t.$$ 
Then for sufficiently small $x>0$, $Q_x(y)$ is a Morse polynomial and the Arnold snake associated to $Q_x(y)$ is the given snake $\sigma$.
\end{theorem*}

What motivates us to construct effectively polynomials in one variable with preassigned critical values configurations is the fact that Davis, Arnold and Douady only proved the existence of such polynomials.

Note that separable permutations were introduced in \cite{BBL}. They are the permutations that do not \enquote{contain} (\cite[page 13, 18]{Gh1}) either of the following two \enquote{patterns}:
 $$\begin{pmatrix}
    1 & 2 & 3& 4 \\
    3 &1& 4& 2
  \end{pmatrix} \text{ or } \begin{pmatrix}
    1 & 2 & 3& 4 \\
    2 &4& 1& 3
  \end{pmatrix}.$$ 
For more information see for instance \cite{BBFGP} or \cite{Ki}. The most recent characterisation of separable permutations is given in \cite{Gh1}, where Ghys describes them in terms of the relative positions of univariate polynomials that change when the polynomials cross a common zero at the origin.

In the next sections we start by introducing in detail the main tools necessary for our construction: Arnold snakes, the contact tree of a finite set of univariate polynomials, and separable permutations. The last section of the paper is dedicated to the main result and its proof.

The context of this work is at the intersection between singularity theory and real algebraic geometry. Interest in the study of real algebraic curves dates back to the works of Harnack, Klein, and Hilbert (\cite{Har1}, \cite{Kl1}, \cite{Hil1}, \cite{Vi1}). In addition, recent considerable progress has been made in this subject, see for instance the results of Ghys (\cite{Gh1}), Itenberg (\cite{II1}), and Sturmfels (\cite{Stu1}).

\section*{Acknowledgements}
This manuscript is based on the first chapter of my PhD thesis (see \cite{So1}), defended at \href{http://math.univ-lille1.fr/}{\textit{Paul Painlevé Laboratory}}. I am very grateful to my PhD advisors, \href{http://math.univ-lille1.fr/~bodin/index.html}{Arnaud Bodin} and \href{http://math.univ-lille1.fr/~popescu/}{Patrick Popescu-Pampu}, for all their guidance throughout this work. During my PhD thesis, my research was financially supported by \href{http://math.univ-lille1.fr/~cempi/}{\textit{Labex CEMPI}} and by \href{http://www.hautsdefrance.fr/}{\textit{the Hauts-de-France Region}}. I would like to express my gratitude to these institutions for their support. The reviewers of my PhD thesis were \href{http://ergarcia.webs.ull.es/}{Evelia R. Garc\' ia Barroso} and \href{https://webusers.imj-prg.fr/~ilia.itenberg/}{Ilia Itenberg}. The president of the committee was \href{http://perso.ens-lyon.fr/ghys/accueil/}{\' Etienne Ghys}.

\section{Arnold snakes}

\begin{definition}\label{def:alternatingSeq}
Let $A:=[a_1,a_2,\ldots,a_{n}]$ be a finite sequence of pairwise distinct real numbers, where $n\geq 1.$ We say that $A$ is \defi{an alternating sequence} if one of the following conditions holds:
 
\begin{enumerate}
\item $a_1>a_2<a_3>a_4<\ldots;$
\item $a_1<a_2>a_3<a_4>\ldots;$
\end{enumerate}

\end{definition}

Let us present a related notion which was used by Arnold (see \cite{Ar1}, \cite{CN}):
\begin{definition}\label{def:snake}
Consider a permutation $\sigma:\{1,2,\ldots,{n}\}\rightarrow \{1,2,\ldots,{n}\}.$ We say that $\sigma$ is \defi{an $n$-snake} if $[\sigma(1),\sigma(2),\ldots,\sigma(n)]$ is an alternating sequence in the sense of Definition \ref{def:alternatingSeq}.
\end{definition}

\begin{remark*}\label{rem:DrawSnake}
For a given snake $\sigma:=\begin{pmatrix}
    1 & 2 &\ldots& n \\
    \sigma(1) & \sigma(2) &\ldots& \sigma(n)
  \end{pmatrix}
  $, consider the set  $\{(i,\sigma(i))\mid i=1,\ldots,n\}$ of points in the real plane. Then for each $i=1,\ldots,n-1$ connect the consecutive points $(i,\sigma(i))$ and $(i+1,\sigma(i+1))$ by an edge. This is the graphic representation of snakes that we use throughout this paper (see for instance Figure \ref{Fig:AltTOSnake}).
\end{remark*}

Note that Andr\'e (\cite{An2}, \cite{An1}) and Stanley (\cite{St}) were also interested in snakes and they were calling them \enquote{alternating permutations}.

\begin{notation}
Denote by $\mathcal{A}_{n}$ the set of alternating sequences of $n$ real numbers and by $\mathcal{S}_{n}$ the set of $n$-snakes.
\end{notation}

Our goal will be to associate snakes to Morse polynomials via the alternating sequence of their critical values.

\begin{definition}\label{def:varphi}
Let $A:=[a_1,\ldots,a_{n}]\in\mathcal{A}_{n}$ be an alternating sequence of $n$ distinct real numbers $a_i\in\mathbb{R}$. Consider a second sequence $A'$ obtained by reordering the elements of $A$ in a strictly increasing way. Define the \textbf{rank of $a_i$}, denoted $\mathrm{rk}(a_i)$, to be the position of $a_i$ in this strictly increasing sequence $A'$. 
The \textbf{snake $\varphi_A$ of $A$} is defined by 
$$\varphi_A(i):=\mathrm{rk}(a_i).$$ Now we can define the surjective function $\varphi:\mathcal{A}_{n}\rightarrow \mathcal{S}_{n}$ as follows: $\varphi(A):=\varphi_A$.
\end{definition}

\begin{example}\label{Ex:AltTOsnake}
Consider the following alternating sequence: $[3>1<4>-1<6].$ See Figure \ref{Fig:AltTOSnake}. 

\begin{figure}[H]
\begin{center}
\begin{tikzpicture}[scale=0.4]
\draw[->] (-1.8,0) -- (6.8,0) node[below] {$x$};
      \draw[->] (0,-1.8) -- (0,6.8) node[left] {$y$};
\tikzstyle{vertex} = [draw,red]
\tikzstyle{edge} = [draw,dashed]

\coordinate (A) at (1,3);
\coordinate (B) at (2,1);
\coordinate (C) at (3,4);
\coordinate (M) at (4,-1);
\coordinate (N) at (5,6);
\coordinate (O) at (0,-0.3);

\coordinate (D) at (1,0);
\coordinate (E) at (2,0);
\coordinate (F) at (3,0);
\coordinate (P) at (4,0);
\coordinate (Q) at (5,0);
\coordinate (W) at (0,6);

\coordinate (Z) at (0,-1);
\coordinate (G) at (0,1);
\coordinate (H) at (0,2);
\coordinate (I) at (0,3);
\coordinate (R) at (0,4);
\coordinate (S) at (0,5);

\node at (Z) [left] {$-1$};
\node at (O) [left] {$O$};
\node at (D) [below] {$1$};
\node at (E) [below] {$2$};
\node at (F) [below] {$3$};
\node at (P) [above] {$4$};
\node at (Q) [below] {$5$};
\node at (W) [left] {$6$};

\node at (G) [left] {$1$};
\node at (H) [left] {$2$};
\node at (I) [left] {$3$};
\node at (R) [left] {$4$};
\node at (S) [left] {$5$};

\fill[vertex] (A) circle(4pt);
\fill[vertex] (B) circle(4pt);
\fill[vertex] (C) circle(4pt);
\fill[vertex] (M) circle(4pt);
\fill[vertex] (N) circle(4pt);

\tikzstyle{vertex} = [draw,black]
\fill[vertex] (D) circle(2pt);
\fill[vertex] (E) circle(2pt);
\fill[vertex] (F) circle(2pt);
\fill[vertex] (P) circle(2pt);
\fill[vertex] (Q) circle(2pt);

\fill[vertex] (G) circle(2pt);
\fill[vertex] (H) circle(2pt);
\fill[vertex] (I) circle(2pt);
\fill[vertex] (R) circle(2pt);
\fill[vertex] (S) circle(2pt);
\fill[vertex] (Z) circle(2pt);
\fill[vertex] (W) circle(2pt);

\draw[edge] (B)--(G);
\draw[edge] (B)--(E);

\draw[edge] (D)--(A);
\draw[edge] (I)--(A);

\draw[edge] (F)--(C);
\draw[edge] (R)--(C);

\draw[edge] (P)--(M);
\draw[edge] (Z)--(M);

\draw[edge] (Q)--(N);
\draw[edge] (W)--(N);

\tikzstyle{edge} = [draw, red, thick]
\draw[edge] (B)--(A);
\draw[edge] (B)--(C);
\draw[edge] (C)--(M);
\draw[edge] (M)--(N);

\begin{scope}[xshift=10cm]
\draw[->] (-1.8,0) -- (6.8,0) node[below] {$i$};
      \draw[->] (0,-1.8) -- (0,6.8) node[left] {$\sigma(i)$};
\tikzstyle{vertex} = [draw,blue]
\tikzstyle{edge} = [draw,dashed]
\coordinate (A) at (1,3);
\coordinate (B) at (2,2);
\coordinate (C) at (3,4);
\coordinate (M) at (4,1);
\coordinate (N) at (5,5);
\coordinate (O) at (0,-0.3);

\coordinate (D) at (1,0);
\coordinate (E) at (2,0);
\coordinate (F) at (3,0);
\coordinate (P) at (4,0);
\coordinate (Q) at (5,0);

\coordinate (G) at (0,1);
\coordinate (H) at (0,2);
\coordinate (I) at (0,3);
\coordinate (R) at (0,4);
\coordinate (S) at (0,5);
\node at (O) [left] {$O$};
\node at (D) [below] {$1$};
\node at (E) [below] {$2$};
\node at (F) [below] {$3$};
\node at (P) [below] {$4$};
\node at (Q) [below] {$5$};

\node at (G) [left] {$1$};
\node at (H) [left] {$2$};
\node at (I) [left] {$3$};
\node at (R) [left] {$4$};
\node at (S) [left] {$5$};

\fill[vertex] (A) circle(4pt);
\fill[vertex] (B) circle(4pt);
\fill[vertex] (C) circle(4pt);
\fill[vertex] (M) circle(4pt);
\fill[vertex] (N) circle(4pt);

\tikzstyle{vertex} = [draw,black]
\fill[vertex] (D) circle(2pt);
\fill[vertex] (E) circle(2pt);
\fill[vertex] (F) circle(2pt);
\fill[vertex] (P) circle(2pt);
\fill[vertex] (Q) circle(2pt);

\fill[vertex] (G) circle(2pt);
\fill[vertex] (H) circle(2pt);
\fill[vertex] (I) circle(2pt);
\fill[vertex] (R) circle(2pt);
\fill[vertex] (S) circle(2pt);

\draw[edge] (B)--(E);
\draw[edge] (B)--(H);

\draw[edge] (D)--(A);
\draw[edge] (I)--(A);

\draw[edge] (F)--(C);
\draw[edge] (R)--(C);

\draw[edge] (P)--(M);
\draw[edge] (G)--(M);

\draw[edge] (Q)--(N);
\draw[edge] (S)--(N);

\tikzstyle{edge} = [draw, blue, thick]
\draw[edge] (B)--(A);
\draw[edge] (B)--(C);
\draw[edge] (C)--(M);
\draw[edge] (M)--(N);
\end{scope}

\end{tikzpicture}
\end{center}
\caption{From an alternating sequence (left) to its snake $\sigma$ (right), as in Example \ref{Ex:AltTOsnake}.
\label{Fig:AltTOSnake}}
\end{figure}
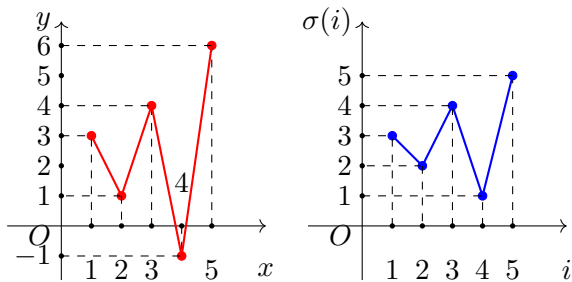

Thus the set of the real numbers to be ordered is $\{3,1,4,-1,6\}.$ We obtain $\mathrm{rk}(3)=3,$ $\mathrm{rk}(1)=2$, $\mathrm{rk}(4)=4$, $\mathrm{rk}(-1)=1,$ $\mathrm{rk}(6)=5$. 
Hence, the $5$-snake associated to $[3>1<4>-1<6]$ is 
$\sigma:=\begin{pmatrix}
    1 & 2 & 3 & 4 & 5\\
    3 & 2 & 4 & 1 & 5
  \end{pmatrix}
  $.   
\end{example}

\begin{definition}\cite[page 64]{La}\label{def:MorsePolyn}
Let us fix $n\geq 1$. We say that a polynomial $P:\bR\rightarrow\bR$ is an $n$-\defi{Morse polynomial} if it satisfies the following conditions:
\begin{enumerate}
\item $\deg P=n$;
\item $P$ is monic, i.e. the leading coefficient of $P$ is equal to $1$;
\item its critical points (i.e. the values $x_i\in\mathbb{C}$ such that $P'(x_i)=0$) are distinct (i.e. simple) and real;
\item its critical values (i.e. the values $P(x_i)$, where $x_i$ is a critical point) are all distinct.
\end{enumerate}
\end{definition}

\begin{proposition}\label{prop:MinMaxAlt}
Between any two consecutive local maxima (respectively minima) of a polynomial function $P:\mathbb{R}\rightarrow \mathbb{R}$ there exists a unique local minimum (respectively maximum) of $P$. In other words, the minima and maxima alternate. 
\end{proposition}

We leave the proof of Proposition \ref{prop:MinMaxAlt} to the reader. 

The following definition is similar to the one given in \cite[pages 66-67]{La}.
\begin{definition}\label{def:snakeOfP}
Let $P:\bR\rightarrow\bR$ be an $n$-Morse polynomial. The $(n-1)$-snake associated to the alternating sequence of the critical values of $P$ (see Definition \ref{def:varphi}) is called \defi{the Arnold snake associated to the polynomial} $P$.
\end{definition}

\begin{remark*}
Since $P$ is monic, the last critical value is a local minimum both in the case where $P$ has an even degree $n$, and in the case where $P$ has an odd degree. Nonetheless, if we also consider non-monic polynomials, for instance with the leading term being $-x^{n},$ the last critical value could be a local maximum. Without any loss of generality, throughout the paper we focus only on monic polynomials.
\end{remark*}

\section{The contact tree}
A very important role in our construction will be played by a combinatorial object called the contact tree, associated to a finite set of univariate polynomials.
\subsection{Standard vocabulary for graphs and trees}

Let us first introduce briefly standard terminology we use for graphs and trees. The reader who is already familiar with these notions is invited to read directly Subsection \ref{subs:notations}. 

Intuitively, a graph represents a collection of vertices and of edges such that each edge connects two vertices. For a detailed description of graphs and for the basic terminology, the reader is invited to refer to \cite{Go}, which represents the main source for the standard notions we introduce and use in the paper. Other helpful sources are for example \cite{Di} or \cite{PP1}.

\begin{definition}\label{DefGraph}\cite{Di,Go}
A \defi{graph} $G=(V,E,\varphi)$ consists of a non-empty finite set of \defi{vertices} $V$, a finite set of \defi{edges} $E$ and an incidence function $\varphi$ that maps edges to unordered pairs of vertices. If the two vertices in the unordered pair are the same vertex, then the edge is called a \defi{loop}. The \defi{valency} of a vertex $v$ is the number of edges of $v$.
\end{definition}

We will be particularly interested in the following class of simple graphs:
\begin{definition}\label{DefTree}\cite{Go} 
\defi{A tree} is a connected graph without cycles. A \defi{rooted tree} is a tree endowed with a base vertex, called its  \defi{root}.
\end{definition}

As a consequence of Definition \ref{DefTree}, trees have no loops, nor multiple edges. Theorem \ref{th:charactTree} below can be seen as a characterisation of trees.

\begin{theorem}\label{th:charactTree}\cite{Go}
A connected simple graph is a tree if and only if it contains a unique path between any two vertices.
\end{theorem}

\begin{definition}\cite{Se}
A \defi{geodesic} between two vertices of a tree is the shortest path between the two vertices.
\end{definition}

In this subsection, our convention is to draw the root on the top of the tree.
\begin{example}
A geodesic is shown in Figure \ref{Fig:geodesic} below.
\begin{figure}[H]
\begin{center}
\begin{tikzpicture}[scale=0.6,sibling distance=5em,
  every node/.style = {shape=circle, fill=black,
    draw, align=center}]]
  \node [shape=rectangle,fill=white] {R}
child {  node  {}
    child { node {} }
    child { node [color=green] {}
      child [color=green]{ node [color=green, text=black] {$V_i$}}
        child [color=green]{ node[color=green] {} 
        child [color=green] { node [color=green, text=black]{$V_k$} }
        child  [color=black] { node {} } }}};
\end{tikzpicture}
\end{center}
\caption{Geodesic in a tree, between the vertices $V_i$ and $V_k$.\label{Fig:geodesic}}
\end{figure}
\end{example}

\begin{definition}\label{def:treeVerticesETC}\cite[page 669]{Go}
In a rooted tree, the length (i.e. the number of edges) of the geodesic from the root to a vertex is called the \defi{level} of the vertex. The root is at level zero. A vertex $p$ of a rooted tree is called the \defi{parent} of a vertex $c$ if $p$ and $c$ are adjacent and the level of $c$ is one plus the level of $p$. We say then also that $c$ is a \defi{child} of $p$. We call \defi{ancestor} of a vertex $c$ every vertex on the path from $c$ to the root (excluding $c$ but including the root). The vertex $c$ is said to be a \defi{descendant} of each of its ancestors. In a rooted tree, a vertex is called a \defi{leaf} if it has no children. A vertex that is not a leaf, i.e. it has children, is said to be an \defi{internal vertex}. Let $v$ be a vertex of the rooted tree $\mathcal{T}$. The \defi{subtree of} $\mathcal{T}$ \defi{rooted at} $v$ is the tree consisting of all the descendants of $v$ in $\mathcal{T},$ including $v$ itself. 
\end{definition}

\begin{remark*}
The following definition is inspired from \cite[Section 3]{Tr} and \cite[page 624]{Go}. Note the difference between the notion of planar tree (can be embedded in a real plane) and plane tree (it is embedded in a real plane).
\end{remark*}

\begin{definition}
A \defi{plane tree} is a tree which is embedded in a real oriented plane (i.e. a topological surface which is oriented and homeomorphic to $\mathbb{R}^2$) without edge crossings.
\end{definition}

\begin{remark*}
One can also give the following abstract characterisation of plane trees: plane trees are abstract ordered trees, in the sense that each vertex has a cyclic order of the edges adjacent to it (see \cite[page 308]{Kn}). Any embedding of a rooted tree in an oriented plane gives the order of children for each vertex: from left to right, for instance (see \cite{wi-ord}). Note that, to give an abstract order on a rooted tree is the same as cyclically ordering the set of children of the root and totally ordering the set of children of any other vertex. 
\end{remark*}

\begin{example}
Two different embeddings of the same abstract rooted tree in the real plane
are shown in Figure \ref{fig:embeddingTrees}. 

\begin{figure}[H]
\begin{center}
\begin{tikzpicture}[scale=0.5,sibling distance=5em,
  every node/.style = {shape=circle, fill=black,
    draw, align=center}]]
    \begin{scope}
 \node [shape=rectangle,fill=white] {R}
child { node {}
    child { node {} }
    child { node {}
      child { node {}
        child { node {} }
        child { node {} }
        child { node {} } }
      child { node {} } }};
      \end{scope}
      \begin{scope}[xshift=7cm]
      \node [shape=rectangle,fill=white] {R}
child { node {}
    child { node {} }
    child { node {}
    child { node {} }
      child { node {}
        child { node {} }
        child { node {} }
        child { node {} } }
       }};
      \end{scope}
\end{tikzpicture}
\end{center}
\caption{Two different embeddings of the same abstract rooted tree.\label{fig:embeddingTrees}}
\end{figure}
\end{example}

\begin{definition}\label{def:equivEmbedd}\cite[page 410]{AF}
Let $\mathcal{T}_1$ and $\mathcal{T}_2$ be two embeddings of the same tree in a real oriented plane $\mathcal{P}$. We say that the two embeddings are \defi{equivalent} if there exists an orientation preserving homeomorphism $\phi$ from $\mathcal{P}$ to $\mathcal{P}$ such that $\phi(\mathcal{T}_1)=\mathcal{T}_2$ and the image of the root of $\mathcal{T}_1$ is the root of $\mathcal{T}_2$.
\end{definition}

Definition \ref{DefcompleteTree} is inspired from \cite{Go}, \cite{St}, and \cite{St1}.

\begin{definition}\label{DefcompleteTree}
A \defi{binary tree} is a rooted tree in which every vertex has at most two children. A plane binary tree is called \defi{complete} if the root and each of the internal vertices has exactly two children.
\end{definition}

\begin{definition}\cite[Definition 3.6]{GBGPPP}\label{Def:rootedCompleteTree}
An \defi{end-rooted rooted tree} is a rooted tree whose root has exactly one neighbour.
\end{definition}

\subsection{Notations}\label{subs:notations}
In this subsection we will set up some useful notations.
\begin{definition}
Let $P\in\mathbb{R}[x]$ and $Q\in\mathbb{R}[x]$ be two polynomials. We say that \defi{the polynomial} $P$ \defi{is smaller than the polynomial} $Q$ \defi{to the right}, denoted by $P\prec_{+} Q$, if and only if one has $P(x)<Q(x)$ for any sufficiently small $0<x\ll 1$, that is to say  there exists $x_0>0$ such that for any $0<x<x_0$, one has $P(x)<Q(x)$.
\end{definition}

Let us consider the monic polynomial $P_x(y)\in\mathbb{R}[x][y],$ such that 
\begin{equation}\label{eq:P}
 P_x(y):=\displaystyle \prod_{i=1}^{m+1}(y-a_i(x))\in \mathbb{R}[x][y],
 \end{equation} 
 where $m\in \mathbb{N}$, the roots $a_i(x)\in\mathbb{R}[x]$ 
 for any $i=1,\ldots,m+1$ and
\begin{equation}\label{eq:rootsIneq}
  a_1(x)\prec_{+} a_2(x)\prec_{+} \ldots \prec_{+} a_{m+1}(x).
  \end{equation} 

Denote by 

\begin{equation}\label{eq:A}
\mathcal{A}:=\{a_1(x),\ldots,a_{m+1}(x)\}
\end{equation}

the set of roots of $P_x(y)$ (see Figure \ref{fig:1}).

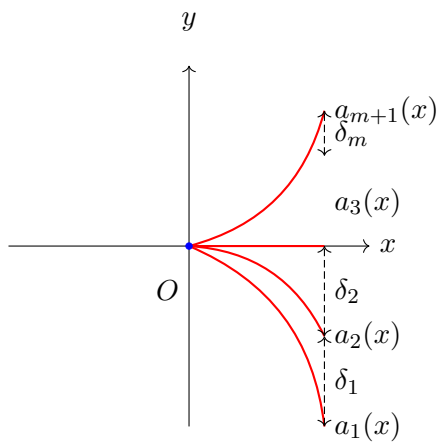
\begin{figure}[H]
\begin{center}
  \begin{tikzpicture}[scale=0.6]


\tikzstyle{vertex} = [draw,blue]
\tikzstyle{edge} = [draw,red,thick]

\coordinate (A) at (0,0);
\coordinate (B) at (3,-4);
\coordinate (C) at (3,0);
\coordinate (D) at (3,3);
\coordinate (E) at (3,-2);

\begin{scope}
      \draw[->] (-4,0) -- (4,0) node[right] {$x$};
      \draw[->] (0,-4) -- (0,4) node[above] {$y$};
      \draw[edge] (A)to[bend left] (B);
      \draw[edge] (A)to[bend left] (E);
      \draw[edge] (A)--(C);
      \draw[edge] (A)to[bend right] (D);
 \draw[<->,line width=0.15mm,dashed,dash pattern=on 1mm off 0.5mm] (3,-4)--(3,-2) 
        node[midway,right] {$\delta_1$};

\draw[<->,line width=0.15mm,dashed,dash pattern=on 1mm off 0.5mm] (3,-2)--(3,-0) 
        node[midway,right] {$\delta_2$};

\draw[<->,line width=0.15mm,dashed,dash pattern=on 1mm off 0.5mm] (3,2)--(3,3) 
        node[midway,right] {$\delta_m$};
\end{scope}

\fill[vertex] (A) circle(2pt);
\node at (A) [below left] {$O$};

\node at (B) [right] {$a_1(x)$};

\node at (E) [right] {$a_2(x)$};

\node at (C) [above right] {$a_3(x)$};

\node at (D) [right] {$a_{m+1}(x)$};

\end{tikzpicture}
  \end{center}
  \caption{The roots of $P_x(y),$ ordered.\label{fig:1}}
\end{figure}

In the sequel, we will consider only polynomials $a_i(x)$ such that  their valuation $\nu_x(a_i(x))\geq 1$, i.e. $a_i(0)=0,$ for any $i=1,\ldots,m+1$. 

\begin{definition}
For any $i=1,\ldots,m$, the difference $$\delta_i(x):=a_{i+1}(x)-a_{i}(x)$$
is called the \defi{gap} between $a_{i+1}(x)$ and $a_{i}(x)$.
\end{definition}

\begin{remark*}
Note that by the initial hypothesis (\ref{eq:rootsIneq}), one has $0\prec_+\delta_i(x),$ for any $i=1,\ldots,m.$
\end{remark*}
\begin{definition}
For any $i=1,\ldots,m,$ the number $$\mathrm{S}_i(x):=\left |  \int_{a_{i}(x)}^{a_{i+1}(x)}P_x(t) \mathrm{d}t \right |$$ is called \defi{the $i$-th area}.
\end{definition}
 
Remember that our goal is to construct Morse polynomials with given Arnold snakes. Roughly speaking, we will do this by controlling the configuration of areas $\mathrm{S}_i(x)$ (see Figure \ref{fig:3} below).

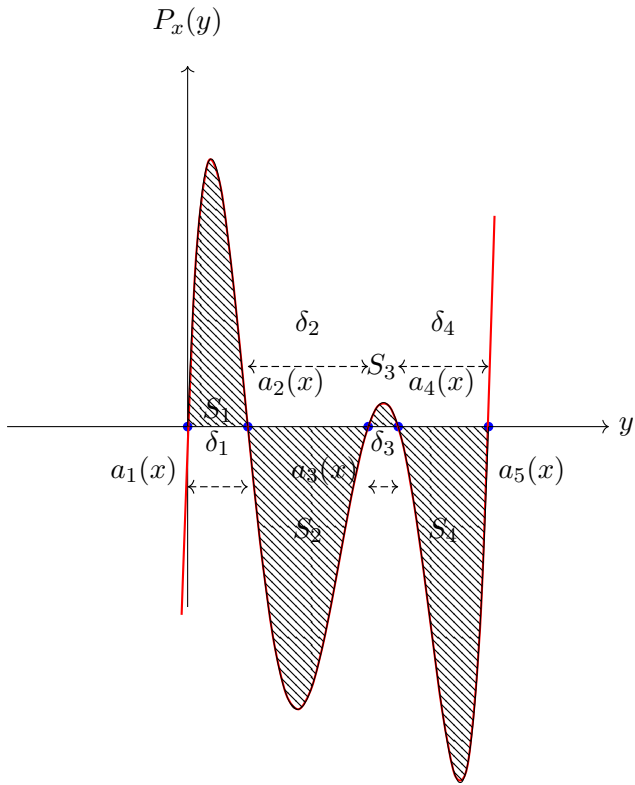
\begin{figure}[H]
\begin{center}
  \begin{tikzpicture}[scale=0.8]


\tikzstyle{vertex} = [draw,blue]
\tikzstyle{edge} = [draw,red,thick]

\begin{scope}
\draw[<->,line width=0.15mm,dashed,dash pattern=on 1mm off 0.5mm] (0,-1)--(1,-1) 
        node[midway,above] {$\delta_1$};

\draw[<->,line width=0.15mm,dashed,dash pattern=on 1mm off 0.5mm] (1,1)--(3,1) 
        node[midway,above] {$\delta_2$};

\draw[<->,line width=0.15mm,dashed,dash pattern=on 1mm off 0.5mm] (3,-1)--(3.5,-1) 
        node[midway,above] {$\delta_3$};
        
\draw[<->,line width=0.15mm,dashed,dash pattern=on 1mm off 0.5mm] (3.5,1)--(5,1) 
        node[midway,above] {$\delta_4$};
\end{scope}

\coordinate (A) at (0,0);
\coordinate (B) at (1,0);
\coordinate (C) at (3,0);
\coordinate (D) at (3.5,0);
\coordinate (E) at (5,0);

\coordinate (X) at (0.5,1);
\coordinate (Y) at (2,-1);
\coordinate (Z) at (2.8,1);
\coordinate (T) at (4.25,-1);

\coordinate (F) at (5.8,2);


\fill[vertex] (A) circle(2pt);
\node at (A) [below left] {$a_1(x)$};

\fill[vertex] (B) circle(2pt);
\node at (B) [above right] {$a_2(x)$};

\fill[vertex] (C) circle(2pt);
\node at (C) [below left] {$a_3(x)$};

\fill[vertex] (D) circle(2pt);
\node at (D) [above right] {$a_4(x)$};

\fill[vertex] (E) circle(2pt);
\node at (E) [below right] {$a_5(x)$};

\node at (X) [below] {$S_1$};
\node at (Y) [below] {$S_2$};
\node at (Z)  {$S_3$};
\node at (T) [below] {$S_4$};


\begin{scope}
      \draw[->] (-3,0) -- (7,0) node[right] {$y$};
      \draw[->] (0,-3) -- (0,6) node[above] {$P_x(y)$};
      \draw[scale=1,domain=-0.1:5.1,
      smooth,variable=\x,thick,red] plot ({\x},{0.5*\x*(\x-1)*(\x-3)*(\x-3.5)*(\x-5)});
      \draw[pattern=north west lines,scale=1,domain=0:1,smooth,variable=\x] plot ({\x},{0.5*\x*(\x-1)*(\x-3)*(\x-3.5)*(\x-5)});

\draw[pattern=north west lines,scale=1,domain=1:3,smooth,variable=\x] plot ({\x},{0.5*\x*(\x-1)*(\x-3)*(\x-3.5)*(\x-5)});

\draw[pattern=north west lines,scale=1,domain=3:3.5,smooth,variable=\x] plot ({\x},{0.5*\x*(\x-1)*(\x-3)*(\x-3.5)*(\x-5)});

\draw[pattern=north west lines,scale=1,domain=3.5:5,smooth,variable=\x] plot ({\x},{0.5*\x*(\x-1)*(\x-3)*(\x-3.5)*(\x-5)});
\end{scope}

\end{tikzpicture}
  \end{center}
  \caption{A given configuration of areas $\mathrm{S}_i(x)$.\label{fig:3}}
\end{figure}

\begin{definition}
Let us consider a sufficiently small $0<x\ll 1$. From now on, $e_i$ stands for the valuation of $\delta_i(x),$ namely $$e_i:=\nu_x(\delta_i(x))\in \mathbb{N}.$$ 
\end{definition}

We will sometimes denote shortly $\mathrm{S}_i$ instead of $\mathrm{S}_i(x)$, and $a_i$ instead of $a_i(x)$.

\subsection{Definition of the contact tree}
The combinatorial object described in what follows plays a key role in our construction of snakes.

In order to study the type of contact between the polynomials $a_i(x)$ (i.e. the roots of the polynomial $P_x(y)$) for a sufficiently small $0<x\ll 1$, we define a tree with numerical information attached. In the sequel we give a constructive definition of this object, called \emph{the contact tree}.

\begin{remark*}
The notion of contact tree is inspired from \emph{the Eggers-Wall tree} (see \cite[Section 4.3]{GBGPPP}, \cite[page 75, Section 4.2]{Wa}). Our definition of the contact tree agrees with the one given by Ghys in \cite{Gh2}, \cite[pages 27-28]{Gh1}, and with the one in \cite[Section 3, Definition 3.5, pages 131-132]{Ka}, where Kapranov called it \emph{the Bruhat-Tits tree}. 
\end{remark*}

Let us first start with an example that makes Definition \ref{def:CT} easier to understand.

\begin{example}
Given two polynomials $$a_1:=m_0 x^0 + m_1 x^1 +\ldots + m_k x^k +\ldots $$ and $$a_{2}:=n_0 x^0 + n_1 x^1 +\ldots + n_k x^k +\ldots,$$ where $m_i=n_i,$ for $i=1,\ldots,k-1$ and $m_k < n_k$, let us construct their contact tree (see Figure \ref{fig:Contractie} below). The relation $m_k < n_k$ gives the order of the two polynomials: $a_1\prec_+ a_2.$

\begin{figure}[H]
\begin{center}
\tikzset{cross/.style={cross out, draw=black, fill=none, minimum size=2*(#1-\pgflinewidth), inner sep=0pt, outer sep=0pt}, cross/.default={2pt}}
\begin{tikzpicture}[scale=0.9]

\tikzstyle{vertex} = [draw,blue]
\tikzstyle{edge} = [draw,red,thick]

\coordinate (A) at (0,0);
\coordinate (B) at (1,0);
\coordinate (C) at (5,0);
\coordinate (D) at (6,0);
\coordinate (E) at (0,1);
\coordinate (F) at (1,1);
\coordinate (G) at (5,1);
\coordinate (H) at (6,1);
\coordinate (FF) at (1,1.2);
\coordinate (GG) at (5,1.2);
\coordinate (HH) at (6,1.2);
\coordinate (I) at (0,2);
\coordinate (J) at (1,2);
\coordinate (K) at (5,2);
\coordinate (L) at (6,2);
\coordinate (JJ) at (1,1.8);
\coordinate (KK) at (5,1.8);
\coordinate (LL) at (6,1.8);

\fill[vertex] (A) circle(2pt);
\node at (A) [below left] {$x^0$};

\fill[vertex] (B) circle(2pt);
\node at (B) [below] {$x^1$};

\fill[vertex] (C) circle(2pt);
\node at (C) [below] {$x^{k-1}$};

\fill[vertex] (D) circle(2pt);
\node at (D) [below] {$x^k$};

\fill[vertex] (E) circle(2pt);
\node at (E) [ left] {$m_0$};

\fill[vertex] (F) circle(2pt);
\node at (FF) [below] {$m_1$};

\fill[vertex] (G) circle(2pt);
\node at (GG) [below] {$m_{k-1}$};

\fill[vertex] (H) circle(2pt);
\node at (HH) [below] {$m_{k}$};

\fill[vertex] (I) circle(2pt);
\node at (I) [ left] {$n_{0}$};

\fill[vertex] (J) circle(2pt);
\node at (JJ) [above] {$n_{1}$};

\fill[vertex] (K) circle(2pt);
\node at (KK) [above] {$n_{k-1}$};

\fill[vertex] (L) circle(2pt);
\node at (LL) [above] {$n_k$};

\fill [pattern=vertical lines] (0,1) rectangle (6,2);

\begin{scope}
      \draw[->] (-3,0) -- (7,0) node[right] {$x$};
      \draw[->] (0,-1) -- (0,4.2) node[left] {$y$};
       \draw[->, color=red] (0,1) -- (7,1) node[right] {$a_1$};
       \draw[->, color=red] (0,2) -- (7,2) node[right] {$a_2$};
        \node at (0,-0.3) [above left] {$O$};
\end{scope}

\begin{scope}[yshift=-5]
      \draw[->] (-3,-4) -- (7,-4) node[right] {$x$};
      \draw[->] (0,-5) -- (0,-1) node[left] {$y$};
       \draw[ color=red] (0,-3) -- (6,-3) ;
        \draw[->, color=red] (6,-3) -- (8,-2)node[right] {$a_2$} ;
         \draw[->, color=red] (6,-3) -- (8,-4) node[right] {$a_1$};
         \fill[vertex] (6,-3) circle(2pt);
        \node at (6,-3) [above] {$k$};
        \node at (0,-4.3) [above left] {$O$};
\node[cross] at (0,-3){}; 
\end{scope}
\end{tikzpicture}

\caption{Construction of the contact tree of two polynomials.\label{fig:Contractie}}
\end{center}
\end{figure}
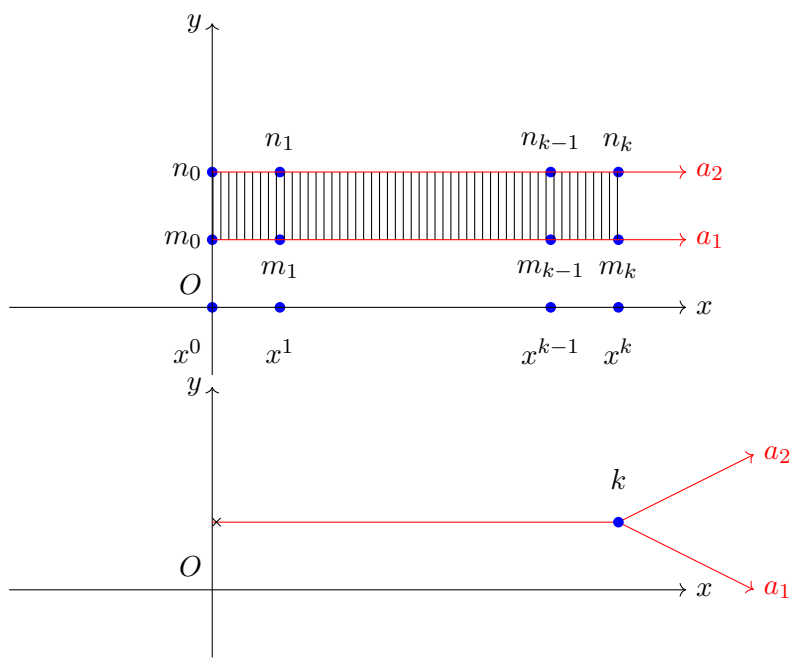

The construction consists of identifying the points of the real plane up to the point where the contact of the corresponding polynomials  ends and the respective coefficients of the two polynomials are starting to differ.

More precisely, if we consider the $xOy$ coordinate system in the real plane, then let each polynomial $a_i(x)$, $i=1,2$ be represented by the semi-straight line $(y=i)\cap (x>0)$ in the first quadrant $\mathbb{R}_+^2$. The line is called the $a_i$-axis. 

The point of coordinate $(\ell,0)$ on the $Ox$ axis corresponds to the monomial $x^\ell$. On each $a_i$-axis a point $(\ell,i)$ is decorated with the coefficient corresponding to the monomial $x^\ell,$ as it appears in the polynomial $a_i(x).$ 

In this example, the valuation $\nu_x(a_{2}(x)-a_{1}(x))=k$. Let us now identify the points of the first quadrant situated between $a_{2}(x)$ and $a_{1}(x)$, up to the point $x=k$, where the two polynomials separate.

The vertices of the tree are: the root, the bifurcation vertex and the two leaves. By construction, the contact tree is rooted and embedded in the real plane.

\end{example}

We are now ready to define the contact tree of a set of univariate polynomials:

\begin{definition}\label{def:CT}
Given the polynomials $a_i(x)\in \mathcal{A}$ (see notation \ref{eq:A}) such that $a_i(0)=0,$ for any $i=1,\ldots,m+1$, we construct \defi{the contact tree} associated to all the polynomials in $\mathcal{A}$ as follows. 

Let us consider the $xOy$ coordinate system in the real plane. In the first quadrant $\mathbb{R}_+^2$, let each polynomial $a_i(x)$ be represented by the semi-straight line $(y=i)\cap (x>0)$. We call it the $a_i$-axis. Recall that, by hypothesis, the polynomials are ordered: $a_1\prec_{+} a_2\prec_{+} \ldots \prec_{+} a_{m+1}.$

The point of coordinate $(k,0)$ on the $Ox$ axis corresponds to the monomial $x^k$. On each $a_i$-axis a point $(k,i)$ is decorated with the coefficient corresponding to the monomial $x^k,$ as it appears in the polynomial $a_i(x).$ 

Note that for each $i=1,\ldots,m$, the valuation $\nu_x(a_{i+1}(x)-a_{i}(x))$ measures the contact between the polynomials $a_{i+1}(x)$ and $a_{i}(x)$. In particular, it is the exponent of maximal contact between $a_{i+1}(x)$ and $a_{i}(x)$. 

For each $i$ let us identify the points of the first quadrant situated between $a_{i+1}(x)$ and $a_{i}(x)$, up to the point where the two polynomials separate. Formally, for each $i=1,\ldots,m$ let us consider the equivalence relation:
$(x_1,y_1) \sim (x_2,y_2) \Leftrightarrow x_1=x_2 $ and $ x_1, x_2 \leq \nu_x(a_{i+1}(x)-a_{i}(x)),$ for any two points $(x_1,y_1)$ and $(x_2,y_2)$ with $i\leq y_1,y_2\leq i+1.$ 

Now we are able to construct inductively the contact tree, by applying the equivalence relation described above, for all $i=1,\ldots,m$. We denote by $\mathcal{CT}(\mathcal{A})$ the quotient given by this equivalence relation and let us call $\Lambda:\mathbb{R}_+^2\rightarrow \mathcal{CT}(\mathcal{A})$ the quotient map. By construction, the quotient $\mathcal{CT}(\mathcal{A})$ is embedded in the real plane. This is due to the fact that the orientation of the contact tree, both horizontally and vertically, is induced by the orientation of the real line $\mathbb{R}$, since the polynomials $a_i(x)$ are  ordered increasingly as horizontal semi-straight lines with their ends on the $Oy$-axis, and the valuations of their corresponding monomials increase along the $Ox$-axis.

In addition, one has $\Lambda(0,0)\equiv \Lambda(0,i),$ for any $i.$ Therefore $\mathcal{CT}(\mathcal{A})$ is a rooted tree. Its root is $\Lambda(0,0).$ 

We call \defi{bifurcation vertices of the contact tree}, the vertices whose valency is greater or equal to $3$. 
 We preserve as numerical decoration only the values of the valuation function in its bifurcation vertices $\nu_x(a_{i+1}(x)-a_{i}(x))$, for any $i=1,\ldots,m.$ Namely, each bifurcation vertex of the contact tree is the point where the contact of the corresponding polynomials $a_i(x)$ and $a_{i+1}(x)$ ends and the respective coefficients of the two polynomials are starting to differ.  For a bifurcation vertex $v$, we shall denote this by $\nu_x(v):=\nu_x(a_{i+1}(x)-a_{i}(x)).$

The vertices of the contact tree are: the bifurcation vertices, the leaves and the root.

The leaves, i.e. the extremities of the maximal geodesics going from the root, are decorated with arrows and are in a bijective correspondence with the polynomials $a_i(x)\in\mathcal{A}$. In other words, each leaf of the tree is decorated with and denoted by its corresponding polynomial $a_i(x).$ For all $i$, denote by $\mathcal{G}_i$ the geodesic from the root to $a_i(x)$. 

The decorating valuations associated to the bifurcation vertices are, by construction, increasing along any geodesic which goes from the root towards the leaves.

By construction, the contact tree is rooted and embedded in the real plane. This induces an orientation on its leaves: its leaves are totally ordered, and the order is given by the vertical order on the $Oy$-axis.
\end{definition}

\begin{example}
Let us give an example of a non binary contact tree: $$\mathcal{CT}(0,x^1,x^1+x^2,\frac{3}{2}x^1+x^2,\frac{3}{2}x^1+x^2+x^3),$$ presented in Figure \ref{Fig:NongenericCtcTree} below. Namely, given $a_1(x)=0,$ $a_2(x)=x^1,$ $a_3(x)=x^1+x^2,$ $a_4(x)=\frac{3}{2}x^1+x^2,$ $a_5(x)=\frac{3}{2}x^1+x^2+x^3,$ there exist two gaps $\delta_1(x)$ and $\delta_3(x)$, such that $e_1=e_3=1.$ In this situation the vertex $a_1\wedge a_2$ coincides with the vertex $a_3\wedge a_4$ and has valency 4.
\end{example}

\begin{figure}[H]
\begin{center}
\tikzset{cross/.style={cross out, draw=black, fill=none, minimum size=2*(#1-\pgflinewidth), inner sep=0pt, outer sep=0pt}, cross/.default={2pt}}
\begin{tikzpicture}[scale=0.6]

\tikzstyle{vertex} = [draw,blue]
\tikzstyle{edge} = [draw,red,thick]
\draw[->] (-3,-6) -- (7,-6) node[right] {$x$};
\draw[->] (0,-7) -- (0,-1) node[left] {$y$};
\node at (0,-6.3) [above left] {$O$};
\draw[ color=red] (0,-5) -- (1,-5) ;
\node[cross] at (0,-5){}; 

\fill[vertex] (1,-5) circle(6pt);
       \node at (1,-4.7) [below,color=blue] {$1$};
\draw[->, color=red] (1,-5) -- (7,-5)node[right] {$a_1$} ;
\draw[ color=red] (1,-5) -- (2,-4.5) ;
\fill[vertex] (2,-4.5) circle(2pt);
       \node at (2,-4) [below,color=blue] {$2$};
       \draw[->, color=red] (2,-4.5) -- (7,-4)node[right] {$a_2$} ;
       \draw[->, color=red] (2,-4.5) -- (7,-3)node[right] {$a_3$} ;
       
       \draw[ color=red] (1,-5) -- (3,-3) ;
       \fill[vertex] (3,-3) circle(2pt);
       \node at (3,-3.5) [above,color=blue] {$3$};
       
       \draw[->, color=red] (3,-3) -- (7,-2)node[right] {$a_4$} ;
       \draw[->, color=red] (3,-3) -- (7,-1)node[right] {$a_5$} ;

\end{tikzpicture}
 \caption{Example of a non binary contact tree, where the vertex $a_1\wedge a_2$ is the same as the vertex $a_3\wedge a_4$. \label{Fig:NongenericCtcTree}}
 \end{center}
\end{figure}
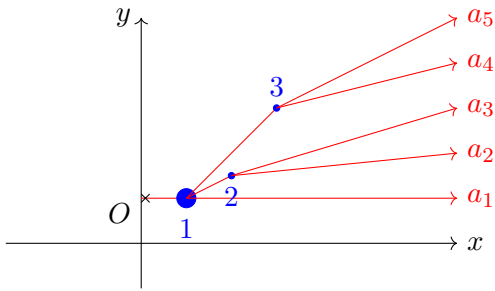

\subsection{Properties of the contact tree}

By construction, the contact tree is a rooted tree, embedded in the real plane, whose leaves are labelled by the polynomials $a_i(x),$ for $i=1,\ldots,m+1$.

The property of being a tree gives the contact tree also the structure of a lower semi-lattice. 

\begin{definition}\cite[page 13]{Vic}
A \defi{partially ordered set} (called also \defi{poset}) is a set endowed with a binary relation that is reflexive, transitive and antisymmetric.
\end{definition}

\begin{definition}\cite[page 14]{Vic}
Let us consider a partially ordered set $S$ and a subset $A\subseteq S.$ Take $b\in S$. We say that $b$ is a \defi{greatest lower bound} for the subset $A$, if the following conditions hold:
\begin{enumerate}
\item the element $b$ is a lower bound for $A$;
\item any other lower bound $c$ for $A$ has the property: $c\leq b.$
\end{enumerate}

\end{definition}

\begin{definition}\cite[page 38]{Vic}
A \defi{lower semi-lattice} is a partially ordered set whose every finite subset has a greatest lower bound.
\end{definition}

\begin{definition}\label{def:ComparatieInContactTree}
Let us consider two vertices $V_1$ and $V_2$ of a rooted tree $\mathcal{T}$. We denote by $\mathcal{G}_1$ the geodesic from the root of $\mathcal{T}$ to $V_1$ and by $\mathcal{G}_2$ the geodesic from the root of $\mathcal{T}$ to $V_2$. If $\mathcal{G}_1\subseteq \mathcal{G}_2$, then we say that $V_1 \leq_{\mathcal{T}} V_2$. In other words, the smallest vertex is the one closer to the root on the geodesic $\mathcal{G}_2.$

If $V_1 \leq_{\mathcal{T}} V_2$ or $V_2 \leq_{\mathcal{T}} V_1$, then we say that $V_1$ and $V_2$ are \defi{comparable vertices} of the tree $\mathcal{T}$. If neither of the two cases mentioned above is true, then one says that $V_1$ and $V_2$ are \defi{not comparable}.
\end{definition}

\begin{proposition}\label{Prop:POSET}
The contact tree is a lower semi-lattice.
\end{proposition}
\begin{proof}
This follows directly from the fact that any rooted tree is a lower semi-lattice (see \cite{Vic,LiGu} for more details).
\end{proof}

We are thus allowed to introduce the following definition: 
\begin{definition}\label{def:meet}
We call the \defi{meet} (or greatest lower bound) of any two distinct vertices $v_i$ and $v_j$ $\in\mathbb{R}[x]$, with $i\neq j$, denoted by $v_i\wedge v_j$ the vertex where the geodesic from the root to $v_i$ and the geodesic from the root to $v_j$ separate in the contact tree. Namely, $v_i\wedge v_j$ is the furthest vertex from the root and it is the most recent common ancestor of $v_i$ and $v_j$. See Figure \ref{fig:5}.
\end{definition}

\begin{figure}[H]
\begin{center}
  \begin{tikzpicture}[scale=0.7,sibling distance=3em,
  every node/.style = {shape=circle, fill=black,
    draw, align=center}, level 1/.style={sibling distance=8em},
  level 2/.style={sibling distance=6em},
  level 3/.style={sibling distance=4em},
  level 4/.style={sibling distance=2em}]
  
   \node [shape=rectangle,fill=white] {R}
child[grow=right, color=red] {node [color=red, label=below:$v_i\wedge v_j$] {}
    child { node [ color=red,label=below :$v_j$]{} child[color=black] { node {} } child [color=black] { node {} }}
    child { node[color=red] {}
      child { node[ color=red,label=above:$v_i$] {}
        child [color=black]{ node {} }
        child [color=black]{ node {} child { node {} } child { node {} }}}
      child [color=black] { node {} } }};
     
\end{tikzpicture}
  \caption{The vertex $v_i\wedge v_j$ where the geodesics of $v_i$ and $v_j$ separate. \label{fig:5}}
  \end{center}
\end{figure}
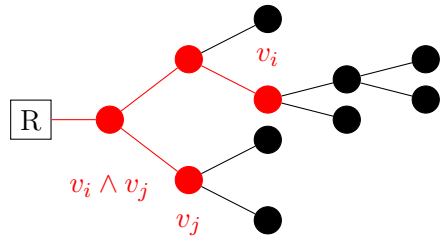

Since the meet in any semilattice is associative as a binary operation (see \cite[page 199]{SS}), one has $(a_i\wedge a_j) \wedge a_k=a_i\wedge (a_j \wedge a_k)$ and we usually skip the parentheses and use the simple notation $a_i\wedge a_j \wedge a_k$.

\begin{remark}\label{remark:ei}
The important thing to note here is the fact that by Definition \ref{def:meet} and the notations introduced before, the meet of two consecutive polynomials $a_{i+1}(x)$ and $a_i(x)\in\mathcal{A}$ has the property $$\nu_x(a_i\wedge a_{i+1})=\nu_x(a_{i+1}(x)-a_i(x))=\nu_x(\delta_i)=e_i.$$
\end{remark}

\begin{proposition}\label{prop:bij-surjection}
The map $(a_i, a_{i+1}) \mapsto a_i \wedge a_{i+1}$ is a  surjection from the set of pairs of consecutive leaves to the set of internal vertices of the contact tree.
 
\end{proposition}
\begin{proof}
By construction of the contact tree, for each internal vertex $v$, called also bifurcation vertex, there is at least a pair of consecutive polynomials $(a_i,a_{i+1}),$ such that the meet of $a_i$ and $a_{i+1}$ is $v$, where $i\in\{1,\ldots,m\}$.
\end{proof}

\begin{lemma}\label{lemma:wedge}
If three polynomials $a_i, a_{i+1}$ and $a_{i+2}\in\mathcal{A}$ (see Subsection \ref{subs:notations}, notation \ref{eq:A}), such that  $a_i\prec_{+} a_{i+1}\prec_{+} a_{i+2}$ are consecutive roots of the polynomial $P\in\mathbb{R}[x][y],$ i.e. three consecutive leaves in the contact tree $\mathcal{CT}(\mathcal{A})$, then $a_i\wedge a_{i+1} \wedge a_{i+2}$ is either $a_i\wedge a_{i+1}$ or $a_{i+1}\wedge a_{i+2}$.
In particular, one has:
$$\nu_x(a_i\wedge a_{i+1} \wedge a_{i+2})=\min(\nu_x(a_i\wedge a_{i+1}), \nu_x(a_{i+1}\wedge a_{i+2})).$$ 
\end{lemma}

\begin{proof}
Recall that the meet in any semilattice is associative (see \cite[page 199]{SS}) as a binary operation, that is to say $(a_i\wedge a_{i+1}) \wedge a_{i+2}=a_i\wedge (a_{i+1}\wedge a_{i+2})=a_i\wedge a_{i+1} \wedge a_{i+2}$. In addition, by definition of the meet $\wedge$, one has both 
\begin{equation}\label{eq:wedge1}
\nu_x(a_i\wedge a_{i+1} \wedge a_{i+2})=\nu_x((a_i\wedge a_{i+1}) \wedge a_{i+2}) \leq \nu_x(a_i\wedge a_{i+1})
\end{equation}
 and 
\begin{equation}\label{eq:wedge2}
\nu_x(a_i\wedge a_{i+1} \wedge a_{i+2})=\nu_x(a_i\wedge (a_{i+1} \wedge a_{i+2}))\leq  \nu_x(a_{i+1} \wedge a_{i+2}).
\end{equation} 

However, there are only the following three possibilities, shown in Figure \ref{fig:4.1} below.

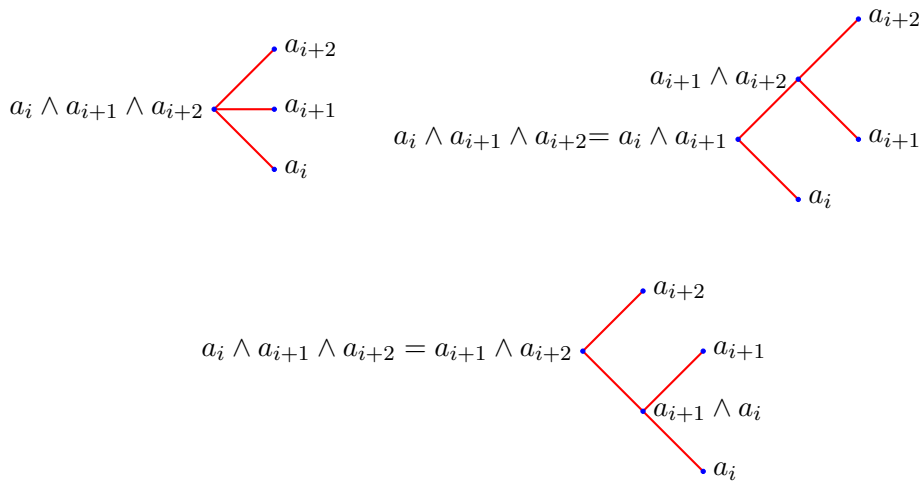
\begin{figure}[H]
  \begin{minipage}{0.3\textwidth}
  \begin{tikzpicture}[scale=0.4]

\tikzstyle{vertex} = [draw,blue]
\tikzstyle{edge} = [draw,red,thick]

\coordinate (A) at (0,0);
\coordinate (B) at (2,2);
\coordinate (C) at (2,0);
\coordinate (D) at (2,-2);

\draw[edge] (A)--(B);
\draw[edge] (A)--(C);
\draw[edge] (A)--(D);

\fill[vertex] (A) circle(2pt);
\node at (A) [left] {$a_i\wedge a_{i+1}\wedge a_{i+2}$};

\fill[vertex] (B) circle(2pt);
\node at (B) [right] {$a_{i+2}$};

\fill[vertex] (C) circle(2pt);
\node at (C) [right] {$a_{i+1}$};

\fill[vertex] (D) circle(2pt);
\node at (D) [right] {$a_{i}$};

\end{tikzpicture}
\end{minipage}
\begin{minipage}{0.3\textwidth}
 \begin{tikzpicture}[scale=0.4]

\tikzstyle{vertex} = [draw,blue]
\tikzstyle{edge} = [draw,red,thick]

\begin{scope}[yshift=-7cm]

\coordinate (E) at (0,0);
\coordinate (F) at (2,2);
\coordinate (G) at (4,4);
\coordinate (H) at (4,0);
\coordinate (I) at (2,-2);

\draw[edge] (E)--(F);
\draw[edge] (E)--(I);
\draw[edge] (F)--(G);
\draw[edge] (F)--(H);

\fill[vertex] (E) circle(2pt);
\node at (E) [left] {$a_i\wedge a_{i+1}\wedge a_{i+2}$\\ $=a_i\wedge a_{i+1}$};

\fill[vertex] (F) circle(2pt);
\node at (F) [left] {$a_{i+1}\wedge a_{i+2}$};

\fill[vertex] (G) circle(2pt);
\node at (G) [right] {$a_{i+2}$};

\fill[vertex] (H) circle(2pt);
\node at (H) [right] {$a_{i+1}$};

\fill[vertex] (I) circle(2pt);
\node at (I) [right] {$a_{i}$};
\end{scope}

\end{tikzpicture}
\end{minipage}

\begin{minipage}{0.3\textwidth}
\begin{center}

  \begin{tikzpicture}[scale=0.4]

\tikzstyle{vertex} = [draw,blue]
\tikzstyle{edge} = [draw,red,thick]

\begin{scope}[yshift=-15cm]

\coordinate (J) at (0,2);
\coordinate (K) at (2,4);
\coordinate (L) at (2,0);
\coordinate (M) at (4,2);
\coordinate (N) at (4,-2);

\draw[edge] (J)--(K);
\draw[edge] (J)--(L);
\draw[edge] (L)--(M);
\draw[edge] (L)--(N);

\fill[vertex] (J) circle(2pt);
\node at (J) [left] {$a_i\wedge a_{i+1}\wedge a_{i+2}= a_{i+1}\wedge a_{i+2}$};

\fill[vertex] (K) circle(2pt);
\node at (K) [right] {$a_{i+2}$};

\fill[vertex] (L) circle(2pt);
\node at (L) [right] {$a_{i+1}\wedge a_i$};

\fill[vertex] (M) circle(2pt);
\node at (M) [right] {$a_{i+1}$};

\fill[vertex] (N) circle(2pt);
\node at (N) [right] {$a_i$};
\end{scope}

\end{tikzpicture}
\end{center}
\end{minipage}
  \caption{The only three possible configurations of $a_i\wedge a_{i+1}\wedge a_{i+2}$.\label{fig:4.1}}
\end{figure}

The conclusion is true for all the three situations above, except for the one presented in Figure \ref{fig:6} below. The case in Figure \ref{fig:6} could be taken into consideration if we were thinking in terms of graphs, but it is impossible in terms of trees, because it creates a cycle inside the contact tree, which is not permitted by Definition \ref{DefTree}. This concludes the proof.
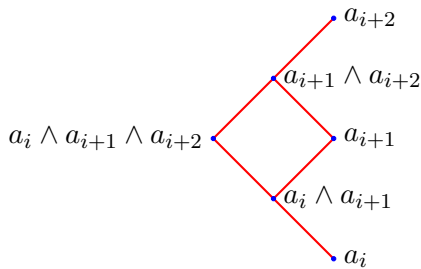
\begin{figure}[H]
  \begin{tikzpicture}[scale=0.4]

\tikzstyle{vertex} = [draw,blue]
\tikzstyle{edge} = [draw,red,thick]

\coordinate (A) at (0,0);
\coordinate (B) at (2,2);
\coordinate (C) at (2,-2);
\coordinate (D) at (4,4);
\coordinate (E) at (4,0);
\coordinate (F) at (4,-4);

\draw[edge] (A)--(B);
\draw[edge] (A)--(C);
\draw[edge] (B)--(D);
\draw[edge] (B)--(E);
\draw[edge] (C)--(E);
\draw[edge] (C)--(F);

\fill[vertex] (A) circle(2pt);
\node at (A) [left] {$a_i\wedge a_{i+1}\wedge a_{i+2}$};

\fill[vertex] (B) circle(2pt);
\node at (B) [right] {$a_{i+1}\wedge a_{i+2}$};

\fill[vertex] (C) circle(2pt);
\node at (C) [right] {$a_i\wedge  a_{i+1}$};

\fill[vertex] (D) circle(2pt);
\node at (D) [right] {$a_{i+2}$};

\fill[vertex] (E) circle(2pt);
\node at (E) [right] {$a_{i+1}$};

\fill[vertex] (F) circle(2pt);
\node at (F) [right] {$a_{i}$};

\end{tikzpicture}
  \caption{Impossible configuration for $a_i\wedge a_{i+1}\wedge a_{i+2}$.  \label{fig:6}}
\end{figure}

\end{proof}

\begin{corollary}\label{cor:p(3)}
If the polynomials $a_i, a_{i+1}$ and $a_{i+2}\in\mathcal{A}$, $a_i\prec_{+} a_{i+1}\prec_{+} a_{i+2}$ are consecutive roots of the polynomial $P_x(y)\in\mathbb{R}[x][y],$ i.e. three consecutive leaves in the contact tree $\mathcal{CT}(\mathcal{A})$, then one has the following equality: $a_i\wedge a_{i+1} \wedge a_{i+2}=a_i\wedge a_{i+2}.$ 
\end{corollary}

\begin{corollary}\label{coroll:GeneralWedge}
Let us consider an arbitrary number of consecutive roots of the polynomial $P\in\mathbb{R}[x][y],$ denoted by $a_i, a_{i+1},\ldots, a_{j}$, $j>i,$. Therefore, $a_i, a_{i+1},\ldots, a_{j}$ are consecutive leaves of the contact tree $\mathcal{CT}(\{a_1, a_2, \ldots, a_{m+1}\})$. Then one has the following equality: 
$$\nu_x(a_i\wedge a_{i+1} \wedge\ldots\wedge a_{j})=\min(\nu_x(a_i\wedge a_{i+1}), \nu_x(a_{i+1}\wedge a_{i+2}),\ldots,\nu_x(a_{j-1}\wedge a_j)).$$ 
Moreover, one also has the following equality: 
$$a_i\wedge a_{i+1} \wedge\ldots\wedge a_{j}=a_i\wedge a_{j}.$$

\end{corollary}

\begin{example}
An example of the conclusions presented in Corollary \ref{coroll:GeneralWedge} can be seen in Figure \ref{Fig:ExGeneralWedge}.

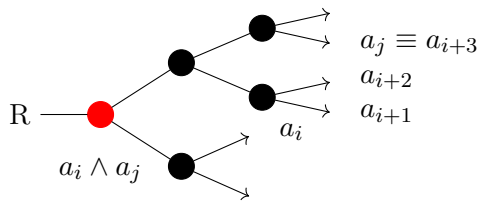
\begin{figure}[H]
  \begin{tikzpicture}[scale=0.6,sibling distance=3em,
   somenode/.style = {shape=circle, scale=0.3,fill=black,
    draw, align=center},
     level 1/.style={sibling distance=8em},
  level 2/.style={sibling distance=6em},
  level 3/.style={sibling distance=4em},
  level 4/.style={sibling distance=2em}]
    \node [shape=rectangle] {R}
child[grow=right] {node [somenode,color=red, label=below:$a_i\wedge a_j$] {}
    child { node [somenode]{} child[color=black] { node {} edge from parent[->]} child [color=black] { node [label=right:$a_{i}$]{} edge from parent[->] }}
    child { node[somenode] {}
      child { node[somenode] {}
        child [color=black]{ node [label=right:$a_{i+1}$]{} edge from parent[->]}
        child [color=black]{ node [label=right:$a_{i+2}$]{}edge from parent[->] }}
      child [color=black] { node [somenode]{} child { node (q9) [label=right:$a_j \equiv a_{i+3}$] {} edge from parent[->] } child { node {} edge from parent[->]}} }};
\end{tikzpicture}
\caption{The maximal contact between several consecutive roots.\label{Fig:ExGeneralWedge}}
\end{figure}
\end{example}

\begin{proof}
For the proof of Corollary \ref{coroll:GeneralWedge}, let us proceed by induction on the number, say $n-1$, of the leaves $a_i, a_{i+1},\ldots,a_{j}.$ Denote by $\mathcal{P}(k)$ the proposition \enquote{if $a_i, a_{i+1},\ldots, a_{j}$ are $k$ consecutive leaves in the contact tree $\mathcal{CT}(\mathcal{A})$, then  
$\nu_x(a_i\wedge a_{i+1} \wedge\ldots\wedge a_{j})=\min(\nu_x(a_i\wedge a_{i+1}), \nu_x(a_{i+1}\wedge a_{i+2}),\ldots,\nu_x(a_{j-1}\wedge a_j)).$}

By Corollary \ref{cor:p(3)}, the proposition $\mathcal{P}(3)$ is true.

Let us prove that if $\mathcal{P}(n-1)$ is true then $\mathcal{P}(n)$ is also true. There are two situations, presented in Figure \ref{fig:recurr1} and Figure \ref{fig:recurr2} below:
\begin{enumerate}
\item either $a_i\wedge \ldots \wedge a_j$ $\leq_{\mathcal{CT}}$ $a_j \wedge a_{j+1}$ and then by the hypothesis $\mathcal{P}(n-1)$ we conclude that $\mathcal{P}(n)$ is also true;
\begin{figure}[H] 
\centering
\includegraphics[scale=0.15]{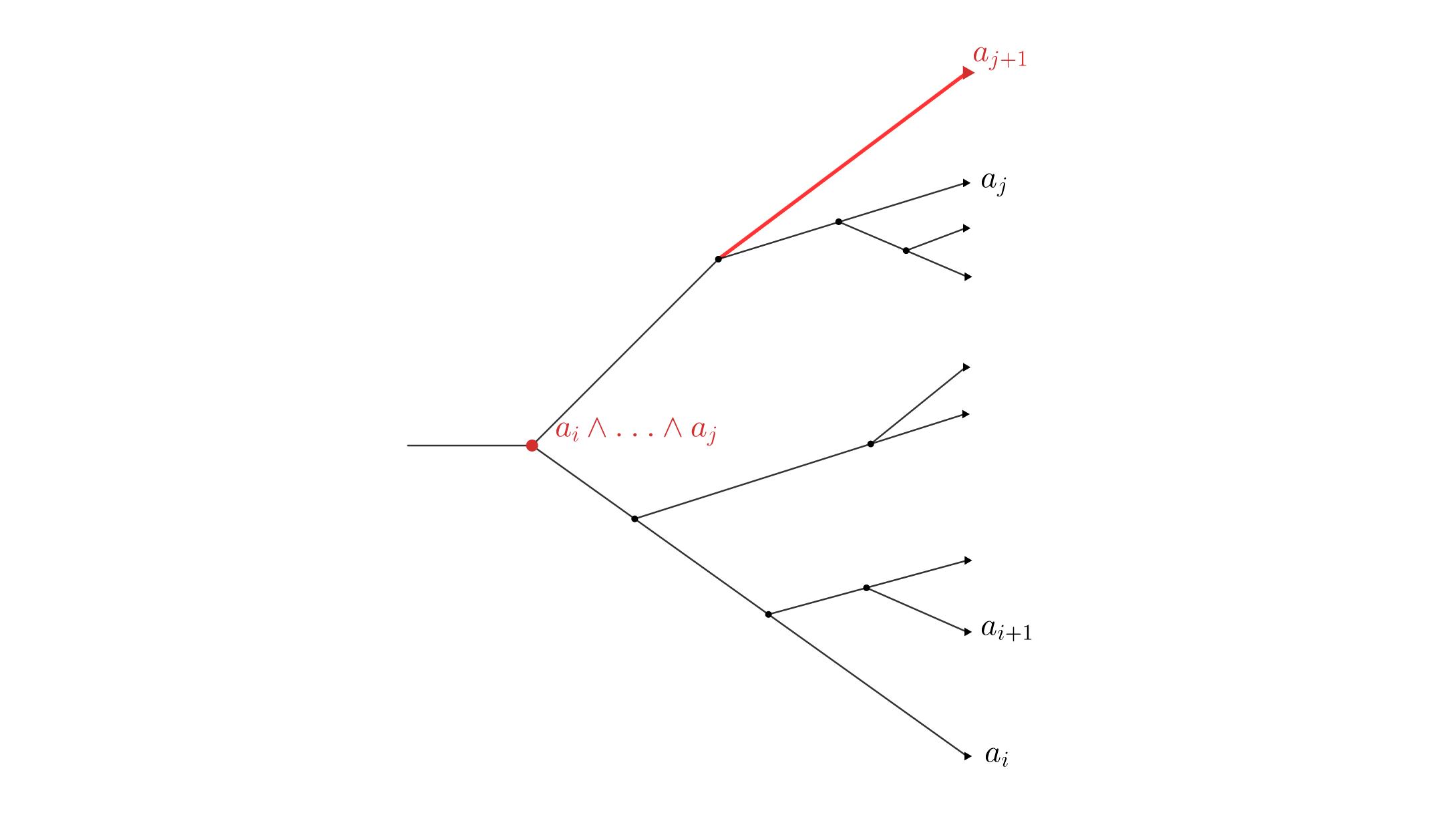}  
\caption{Recursive step case (a).\label{fig:recurr1}}
\end{figure}

\item or $a_i\wedge \ldots \wedge a_j$ $>_{\mathcal{CT}}$ $a_j \wedge a_{j+1}$; in this case, $\nu_x(a_i\wedge a_{i+1} \wedge\ldots\wedge a_{j}\wedge a_{j+1})=\nu_x(a_j\wedge a_{j+1})$. 
Since $a_i\wedge \ldots \wedge a_j$ $>_{\mathcal{CT}}$ $a_j \wedge a_{j+1}$, we have $ \nu_x(a_j \wedge a_{j+1})=\min(\nu_x(a_i\wedge a_{i+1}), \nu_x(a_{i+1}\wedge a_{i+2}),\ldots,\nu_x(a_{j-1}\wedge a_j))$ and the induction step is finished.

\begin{figure}[H] 
\centering
\includegraphics[scale=0.15]{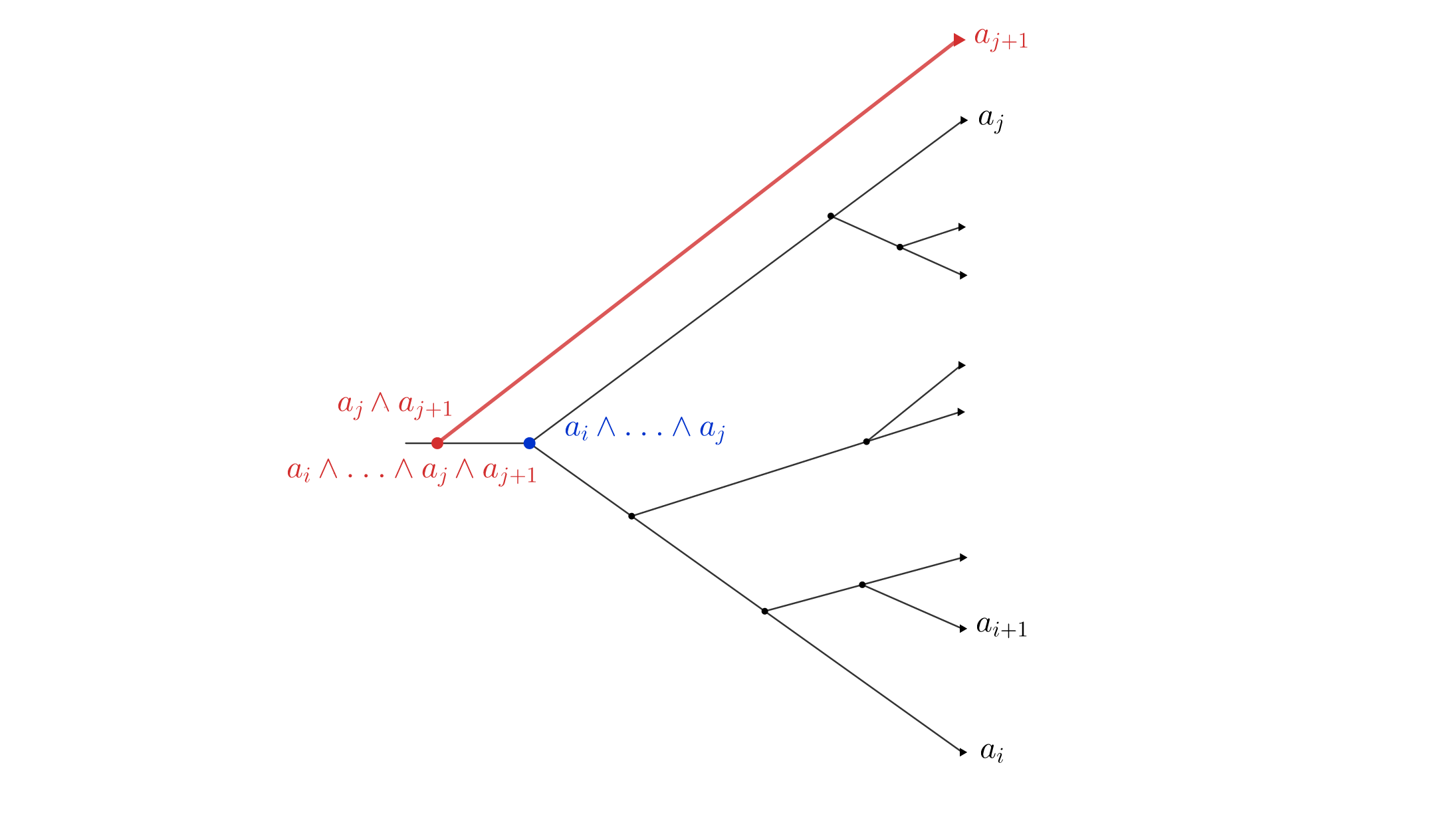} 
\caption{Recursive step case (b).\label{fig:recurr2}}
\end{figure}

\end{enumerate}

\end{proof}

In the sequel we consider only complete plane binary contact trees, in order to have a bijection between the bifurcation vertices and the pairs $(a_{i+1},a_{i})$ for $i=1,\ldots,m$.

\begin{corollary}
If the contact tree is an end-rooted complete plane binary plane tree then there is a bijection between the set of internal vertices of the contact tree and the set of the vertices $a_i\wedge a_{i+1},$ for $i=1,\ldots,m$.
\end{corollary}

\begin{remark}\label{remark:GenericContactTreeBijection}
To be more precise, if the contact tree is complete plane binary, then there is a bijection between the set $\{a_i\wedge a_{i+1}\mid i=1,\ldots, m\}$ and the set of pairs  $\{(a_i,a_{i+1})\mid i=1,\ldots, m-1\}$: to each pair of consecutive polynomials it corresponds the unique vertex $a_i\wedge a_{i+1}$ and vice-versa. In this case, the contact tree is necessarily an end-rooted complete plane binary tree. 
\end{remark}

\section{A valuative study on the contact tree}\label{subs:v(Si)}
This section consists of a valuative study for small enough $x>0$. We start by giving an exact computation of the valuation in $x$ of any area $\mathrm{S}_i(x)$, using the numerical information that decorates the contact tree (see Proposition \ref{prop:valuationOfSi}). This result is an important step towards the main goal of this section: Proposition \ref{prop:valuationOfSumOfAreas}, which will be one of the key ingredients in our construction of separable snakes.

In the sequel, let us fix $i$ and compute the valuation $\nu_x(\mathrm{S}_i(x)),$ in terms of the contact tree and of the valuations $e_j:=\nu_x(\delta_j(x)),$ for $j=1,\ldots,m$.

Recall that the roots of the polynomial $$P_x(y):=\displaystyle \prod_{i=1}^{m+1}(y-a_i(x))\in \mathbb{R}[x][y]$$ (see equation \ref{eq:P}), namely $a_i(x)\in\mathcal{A}$ verify equation \ref{eq:rootsIneq}:
  $$a_1(x)\prec_{+} a_2(x)\prec_{+} \ldots \prec_{+} a_{m+1}(x).$$ For any $i=1,\ldots,m+1$, one has $$\delta_i(x):=a_{i+1}(x)-a_{i}(x),$$ $$\nu_x(\delta_i(x)):=e_i$$ and the area $$\mathrm{S}_i(x):=\left | \int_{a_{i}(x)}^{a_{i+1}(x)}P_x(y) \mathrm{d}y \right |.$$ 
  
  \begin{figure}[H]
  \begin{center}
  \begin{tikzpicture}[scale=0.8]


\tikzstyle{vertex} = [draw,blue]
\tikzstyle{edge} = [draw,red,thick]

\coordinate (A) at (1,0);
\coordinate (B) at (4,0);
\coordinate (C) at (2.5,1);
\coordinate (D) at (2.5,-1);


\fill[vertex] (A) circle(2pt);
\node at (A) [above] {$a_i(x)$};

\fill[vertex] (B) circle(2pt);
\node at (B) [above right] {$a_{i+1}(x)$};

\node at (C) [below] {$S_i(x)$};

\node at (D) [above] {$\delta_i(x)$};


\begin{scope}
      \draw[->] (-3,0) -- (7,0) node[right] {$y$};
      \draw[->] (0,-3) -- (0,2.2) node[above] {$P_x(y)$};
      \draw[scale=1,domain=-0.1:5.1,smooth,variable=\x,thick,blue] plot ({\x},{-0.5*(\x-1)*(\x-4)});
      \draw[pattern=north west lines,scale=1,domain=1:4,smooth,variable=\x] plot ({\x},{-0.5*(\x-1)*(\x-4)});
\end{scope}

\end{tikzpicture}
  \end{center}
  \caption{Area $\mathrm{S_i}$, corresponding to $\delta_i=a_{i+1}-a_i$.  \label{fig:7}}
\end{figure}
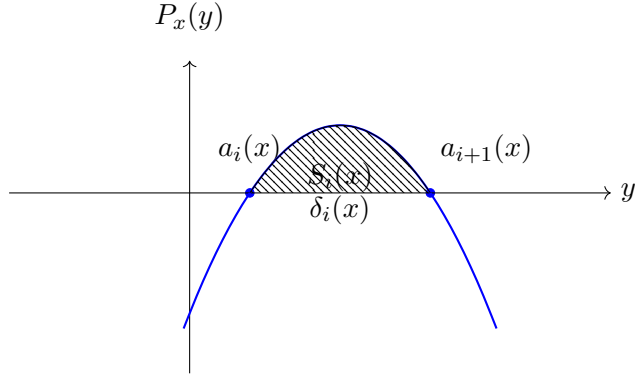

\begin{proposition}\label{prop:valuationOfSi}
Let us consider the geodesic $\mathcal{G}_i$, from the root to $a_i\wedge a_{i+1}$. One has the formula: 

\begin{equation}\label{ec:ValuationS}
\nu_x(\mathrm{S}_i(x))=e_i+\sum_{\{G\in \mathcal{G}_i\mid G\leq_{\mathcal{CT}} a_i\wedge a_{i+1}\}} c_{G}(i)  \nu_x(G),
\end{equation}
 
where the coefficient $c_{G}(i)\in\mathbb{N}$ represents the number of leaves of the contact tree $\mathcal{CT}(\mathcal{A})$, whose most recent ancestor belonging to $\mathcal{G}_i$ is $G$. We shall further write $c_{G}(i):=c_G$ if $i$ is clear from the context. 
\end{proposition}

Before proceeding to the proof of Proposition \ref{prop:valuationOfSi}, we strongly recommend  the lecture of Example \ref{ex:calculsValuationSi} below.

\begin{example}\label{ex:calculsValuationSi}
Let us consider the following given roots of the polynomial $P_x(y):$

$a_1(x)=0,$ 

$a_2(x)=x^4$, 

$a_3(x)=x^3+x^4$, 

$a_4(x)=x^2+x^3+x^4$,

 $a_5(x)=x^1+x^2+x^3+x^4.$ 

Therefore, we obtain: 

$\delta_1:=a_2(x)-a_1(x)=x^4$ and $e_1:=\nu_x(\delta_1)=4$, 

$\delta_2:=a_3(x)-a_2(x)=x^3$ and $e_2:=\nu_x(\delta_2)=3$, 

$\delta_3:=a_4(x)-a_3(x)=x^2$ and $e_3:=\nu_x(\delta_3)=2$,
 
$\delta_4:=a_5(x)-a_4(x)=x^1$ and $e_4:=\nu_x(\delta_4)=1$.
 
We want to compute the valuation of the following area $S_3(x):=\left | \int_{a_3(x)}^{a_4(x)}P_x(y)\mathrm{d}y\right |.$ 
 
We may write $S_3(x)=\int_{a_3(x)}^{a_4(x)}(y-a_1(x))(y-a_2(x))(y-a_3(x))(a_4(x)-y)(a_5(x)-y)\mathrm{d}y.$ 

Consider the following change of variable: $y:=a_3(x)+s\delta_3(x),$ where $s\in[0,1].$ We obtain :

\begin{footnotesize}
$$S_3(x)=\int_{0}^{1}(\delta_1(x)+\delta_2(x)+s\delta_3(x))(\delta_2(x)+s\delta_3(x))s\delta_3(x)(1-s)\delta_3(x)(\delta_4(x)+(1-s)\delta_3(x))\delta_3(x)\mathrm{d}s=$$

$$=\delta_1(x)\delta_2(x)\delta_3(x)^3\delta_4(x)
\int_{0}^1 s(1-s)\mathrm{d}s+\delta_1(x)\delta_2(x)\delta_3(x)^4
\int_{0}^1 s(1-s)^2\mathrm{d}s+\delta_1(x)\delta_3(x)^4\delta_4(x)
\int_{0}^1 s^2(1-s)\mathrm{d}s+$$ $$+\delta_1(x)\delta_3(x)^5
\int_{0}^1 s^2(1-s)^2\mathrm{d}s+\delta_2(x)^2\delta_3(x)^3\delta_4(x)
\int_{0}^1 s(1-s)\mathrm{d}s+\delta_2(x)^2\delta_3(x)^4
\int_{0}^1 s(1-s)^2\mathrm{d}s+$$ $$+\delta_2(x)\delta_3(x)^4\delta_4(x)
\int_{0}^1 2s^2(1-s)\mathrm{d}s+\delta_2(x)\delta_3(x)^5
\int_{0}^1 2s^2(1-s)\mathrm{d}s+\delta_3(x)^5\delta_4(x)
\int_{0}^1 s^3(1-s)\mathrm{d}s+\delta_3(x)^6
\int_{0}^1 s^3(1-s)\mathrm{d}s.$$ 

Thus $$\nu_x(S_3(x))=\mathrm{min}\bigg\{
\nu_x(\delta_1(x)\delta_2(x)\delta_3(x)^3\delta_4(x)),\nu_x(\delta_1(x)\delta_2(x)\delta_3(x)^4),
\nu_x(\delta_1(x)\delta_3(x)^4\delta_4(x)), \nu_x(\delta_1(x)\delta_3(x)^5),
\nu_x(\delta_2(x)^2\delta_3(x)^3\delta_4(x)),$$
$$\nu_x(\delta_2(x)^2\delta_3(x)^4), \nu_x(\delta_2(x)\delta_3(x)^4\delta_4(x)),
\nu_x(\delta_2(x)\delta_3(x)^5), \nu_x(\delta_3(x)^5\delta_4(x)),
\nu_x(\delta_3(x)^6)\bigg\}=\nu_x(\delta_3(x)^5\delta_4(x))=5e_3+e_4=\textbf{11}.$$
\end{footnotesize}
Let us now interpret the result obtained above, this time with respect to the contact tree. Denote by $\mathcal{G}_i$, the geodesic from the root to $a_i\wedge a_{i+1}$.  See Figure \ref{Fig:calculValSi}.

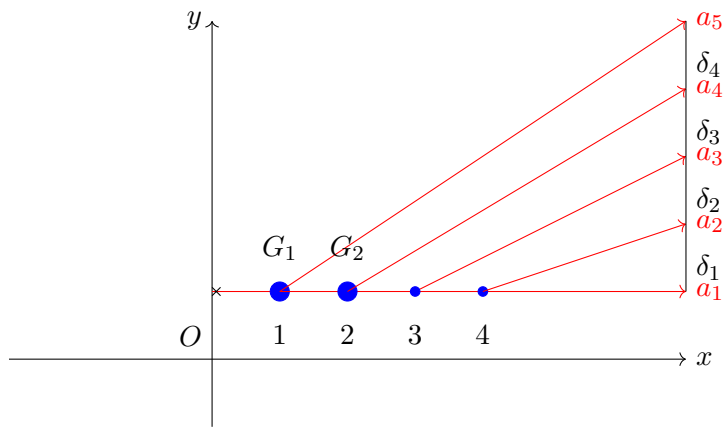
\begin{figure}[H]
\begin{center}
\tikzset{cross/.style={cross out, draw=black, fill=none, minimum size=2*(#1-\pgflinewidth), inner sep=0pt, outer sep=0pt}, cross/.default={2pt}}
\begin{tikzpicture}[scale=0.9]

\tikzstyle{vertex} = [draw,blue]
\tikzstyle{edge} = [draw,red,thick]
\draw[->] (-3,-6) -- (7,-6) node[right] {$x$};
\draw[->] (0,-7) -- (0,-1) node[left] {$y$};
\node at (0,-6.3) [above left] {$O$};
\draw[ color=red] (0,-5) -- (1,-5) ;
\node[cross] at (0,-5){}; 

\fill[vertex] (1,-5) circle(4pt);
       \node at (1,-5) [below] {$1$};
        \node at (1,-5) [above] {$G_1$};
\draw[->, color=red] (1,-5) -- (7,-5)node[right] {$a_1$} ;
\draw[->, color=red] (1,-5) -- (7,-1)node[right] {$a_5$} ;
\fill[vertex] (2,-5) circle(4pt);
       \node at (2,-5) [below] {$2$};
       \node at (2,-5) [above] {$G_2$};
       \draw[->, color=red] (2,-5) -- (7,-2)node[right] {$a_4$} ;
       \draw[->, color=red] (3,-5) -- (7,-3)node[right] {$a_3$} ;
       \fill[vertex] (3,-5) circle(2pt);
       \node at (3,-5) [below] {$3$};
        \fill[vertex] (4,-5) circle(2pt);
       \node at (4,-5) [below] {$4$};
       
       \draw[->, color=red] (4,-5) -- (7,-4)node[right] {$a_2$} ;
       
        \draw (7,-5) -- (7,-4)node[below right] {$\delta_1$} ;
        \draw (7,-4) -- (7,-3)node[below right] {$\delta_2$} ;
         \draw (7,-3) -- (7,-2)node[below right] {$\delta_3$} ;
          \draw (7,-2) -- (7,-1)node[below right] {$\delta_4$} ;

\end{tikzpicture}
 \caption{How to read the valuation of $\mathrm{S}_3(x)$ on the contact tree of the roots $a_i(x)\in\mathcal{A}$.\label{Fig:calculValSi}}
 \end{center}
\end{figure}

We want to check that: 
$$\nu_x(\mathrm{S}_3)=e_3+\sum_{\{G\in \mathcal{G}_3\mid G\leq_{\mathcal{CT}} a_3\wedge a_{4}\}} c_{G}(3)  \nu_x(G),$$

\noindent where the coefficient $c_{G}(3)\in\mathbb{N}$ represents the number of leaves of the contact tree $\mathcal{CT}(\mathcal{A})$, whose most recent ancestor belonging to $\mathcal{G}_3$ is $G$.

In this example,  $a_3\wedge a_{4}\equiv G_2$ and we can check that formula \ref{ec:ValuationS} gives us the same result, namely $\nu_x(S_3(x))=11$: the vertices on the geodesic $\mathcal{G}_3$ which are $\leq_{\mathcal{CT}}G_2$ are $G_1$ and $G_2$. We have $c_{G_1}(3)=1$ representing the number of leaves of the contact tree $\mathcal{CT}(\mathcal{A})$, whose most recent ancestor belonging to $\mathcal{G}_3$ is $G_1$. Also, $c_{G_2}(3)=4$ representing the number of leaves of the contact tree $\mathcal{CT}(\mathcal{A})$, whose most recent ancestor belonging to $\mathcal{G}_3$ is $G_2$. 

Thus $\nu_x(\mathrm{S}_3)=e_3+ c_{G_1}(3)  \nu_x(G_1)+ c_{G_2}(3)  \nu_x(G_2)=2+1\times 1+4\times 2=11.$ The verification is completed.

\end{example}

\begin{proof}
Let us prove Proposition \ref{prop:valuationOfSi}. We have (recall equation \ref{eq:P})
$$\mathrm{S}_i(x):=\left | \int_{a_{i}(x)}^{a_{i+1}(x)}P_x(y) \mathrm{d}y \right |=\left | \int_{a_{i}(x)}^{a_{i+1}(x)} (y-a_1(x))(y-a_2(x))\cdots (y-a_{m+1}(x))\mathrm{d}y\right |.$$ 

Note that $y\geq a_k(x),$ for any $k\leq i$ and that $y\leq a_k(x)$ for any $k\geq i+1$, thus by taking the absolute value we obtain: 
$$\mathrm{S}_i(x)=\int_{a_{i}(x)}^{a_{i+1}(x)} (y-a_1(x))(y-a_2(x))\cdots (y-a_i(x))(a_{i+1}(x)-y)(a_{i+2}(x)-y)\cdots (a_{m+1}(x)-y)\mathrm{d}y.$$

For $s\in[0,1]$ (see Figure \ref{fig:delta} below), make the change of variables $y=a_i(x)+s\delta_i(x).$
 
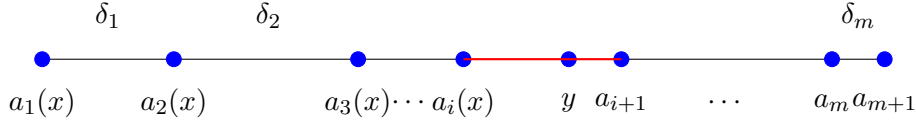
\begin{figure}[H]
\begin{center}
\tikzset{cross/.style={cross out, draw=black, fill=none, minimum size=2*(#1-\pgflinewidth), inner sep=0pt, outer sep=0pt}, cross/.default={2pt}}
\begin{tikzpicture}[scale=0.7]
\tikzstyle{vertex} = [draw,blue]
\tikzstyle{edge} = [draw,red,thick]
\draw (0,0) -- (16,0) ;
\fill[vertex] (0,0) circle(4pt);
       \node at (0,0) [below] {$a_1(x)$};
        \node at (1.25,0) [above] {$\delta_1$};
\fill[vertex] (2.5,0) circle(4pt);
       \node at (2.5,0) [below] {$a_2(x)$};       
        \node at (4.3,0) [above] {$\delta_2$};
       \fill[vertex] (6,0) circle(4pt);
       \node at (6,0) [below] {$a_3(x)$};   
        \fill[vertex] (8,0) circle(4pt);
       \node at (8,0) [below] {$a_i(x)$};  
        \node at (7,0) [below] {$\cdots$};
          \fill[vertex] (10,0) circle(4pt);
       \node at (10,0) [below] {$y$};      
           \fill[vertex] (11,0) circle(4pt);
       \node at (11,0) [below] {$a_{i+1}$};    
       \node at (13,0) [below] {$\cdots$};
        \fill[vertex] (15,0) circle(4pt);
       \node at (15,0) [below] {$a_{m}$};
        \fill[vertex] (16,0) circle(4pt);
        \node at (15.5,0) [above] {$\delta_m$};
       \node at (16,0) [below] {$a_{m+1}$};    
       \draw[ color=red, thick] (8,0) -- (11,0) ;

\end{tikzpicture}
 \caption{Change of variable $y=a_i(x)+s\delta_i(x)$\label{fig:delta}}
 \end{center}
\end{figure}

Therefore:
$$
 \mathrm{S}_i(x)=\int_{0}^{1}(\delta_1(x)+\delta_2(x)+\cdots+
\delta_{i-1}(x)+s\delta_i(x))(\delta_2(x)+\cdots+
\delta_{i-1}(x)+s\delta_i(x))\cdots $$
$$
\cdots (0+s\delta_i(x)) ((1-s)\delta_i(x)) ((1-s)\delta_i(x)+\delta_{i+1}(x))\cdots ((1-s)\delta_i(x)+\delta_{i+1}(x)+\cdots +\delta_m(x)) (\delta_i(x)\mathrm{d}s).
$$

Let us study the valuation in $x$ of  one of the parentheses, say $$\nu:=\nu_x((1-s)\delta_i(x)+\delta_{i+1}(x)+\cdots +\delta_j(x)).$$ The other parentheses can all be treated similarly. 

The key fact is that by integrating in $s$ one does not change the valuation in $x$.

Since $a_1\prec_{+} a_2\prec_{+} \ldots\prec_{+} a_{m+1}$, one has $\delta_i(x)>0,$ for any $0<x\ll 1$ and for any $i=1,\ldots,m.$ Thus we need to compute the valuation of a sum of positive terms. Therefore the valuation will be equal to the minimum of the valuations of the terms in the chosen  parenthesis, namely: $\nu=\min (\nu_x(\delta_i(x)),
\nu_x(\delta_{i+1}(x)),\ldots,\nu_x(\delta_j(x)))$. By Remark \ref{remark:ei} we get $\nu=\min (\nu_x(a_{i}\wedge a_{i+1}),
\nu_x(a_{i+1}\wedge a_{i+2}),\ldots,\nu_x(a_{j}\wedge a_{j+1}))$. By Corollary \ref{coroll:GeneralWedge}, we obtain $\nu=\nu_x(a_i\wedge a_{i+1} \wedge \ldots \wedge a_{j+1}).$ There are two consequences of this fact, as follows:
\begin{enumerate}
\item $\nu\leq \nu_x(a_{i}\wedge a_{i+1})=e_i$;

\item by Corollary \ref{coroll:GeneralWedge}, $\nu=\nu_x(a_i \wedge a_{j+1}),$ 
namely the valuation of the vertex $G_{a_i \wedge a_{j+1}}:=G$ where the geodesic of $a_i$ and the geodesic of $a_{j+1}$ separate in the contact tree. Thus the vertex $G$ is on the geodesic $\mathcal{G}_i.$

\end{enumerate}

For fixed $i$, to each leaf $a_j$, $j=1,\ldots,m+1$, it corresponds a parenthesis in the expression of $\mathrm{S}_i$ as above. As we have just seen, the parenthesis corresponds to exactly one vertex $G$ situated on the geodesic $\mathcal{G}_i$ such that $G\leq_{\mathcal{CT}} a_i\wedge a_{i+1}$. 

Let us count differently: for each vertex $G$ laying on the geodesic $\mathcal{G}_i$, with $ G \leq_{\mathcal{CT}} a_i \wedge a_{i+1}$, we count the number of leaves of $\mathcal{CT}$ whose most recent ancestor belonging to $\mathcal{G}_i$ is $G$. Denote this number by $c_{G}.$ Note that we only consider the vertices $ G \leq_{\mathcal{CT}} a_i \wedge a_{i+1}$ since we proved above that for each parenthesis we have $\nu\leq \nu_x(a_{i}\wedge a_{i+1})=e_i$.  

In addition, since we made the change of variable, another factor $\delta_i(x)$ appeared in the product. Thus the valuation $e_i=\nu_x(\delta_i(x))$ has to be taken into consideration when we sum all the terms. In conclusion, $$\nu_x(\mathrm{S}_i(x))=e_i+\sum_{\{G\in \mathcal{G}_i\mid G \leq_{\mathcal{CT}} a_i\wedge a_{i+1}\}} c_{G}  \nu_x (G).$$

\end{proof}

\begin{example}\label{ex:S_4}
Given the roots $a_1(x)=0,$ $a_2(x)=x^6,$ $a_3(x)=x+x^6,$ $a_4(x)=x+x^2+x^6$, $a_5(x)=x+x^2+x^3+x^6,$ $a_6(x)=x+x^2+x^3+x^4+x^6,$ $a_7(x)=x+x^2+x^3+x^4+x^5+x^6,$ let us compute $\nu_x( \mathrm{S}_i(x)),$ for $i=4.$  

\begin{figure}[H]
  \begin{tikzpicture}[scale=0.3,grow=right,
  mynode/.style={
         draw, rounded corners=2pt
  },
  level 1/.style={sibling distance=20em},
  level 2/.style={sibling distance=20em},
  level 3/.style={sibling distance=12em},
  level 4/.style={sibling distance=11em},
  level 5/.style={sibling distance=10em}]
   \node  [mynode] {$x^0$} 
    child { node  [mynode,fill=yellow]  {$x$} 
		    child { node  [mynode] {$x^6$}            
              child { node  [mynode,fill=yellow]  {$a_1$} 
             edge from parent node {}}
              child { node  [mynode,fill=yellow]  {$a_2$} 
              edge from parent node {}}
             edge from parent node {} }
             child { node  [mynode,fill=green]  {$x^2$} 
             child { node  [mynode,fill=green] {$a_3$} 
             edge from parent node {}  }
             child { node  [mynode,fill=pink] {$x^3$} 
             child { node  [mynode,fill=pink] {$a_4$} 
             edge from parent node {}  }
             child { node  [mynode] {$x^4$} 
             child { node  [mynode,fill=pink] {$a_5$} 
             edge from parent node {}  }
             child { node  [mynode] {$x^5$}
             child { node  [mynode,fill=pink] {$a_6$} 
             edge from parent node {}  }
             child { node  [mynode,fill=pink] {$a_7$} 
             edge from parent node {}  }
             edge from parent node {}  }
             edge from parent node {}  }
             edge from parent node {}  }
             edge from parent node {}  }            
      edge from parent node {} };
      \draw[<->,line width=0.15mm,dashed,dash pattern=on 1mm off 0.5mm] (18,-8)--(18,-1.75) 
        node[midway,right] {$\mathbb{\delta}_1$};
        \draw[<->,line width=0.15mm,dashed,dash pattern=on 1mm off 0.5mm] (18,-1.75)--(18,1.8) 
        node[midway,right] {$\mathbb{\delta}_2$};
        \draw[<->,line width=0.15mm,dashed,dash pattern=on 1mm off 0.5mm] (18,1.8)--(18,4.5) 
        node[midway,right] {$\mathbb{\delta}_3$};
        \draw[<->,line width=0.15mm,dashed,dash pattern=on 1mm off 0.5mm,color=red] (18,4.5)--(18,7) 
        node[midway,right] {\textcolor{red}{$\mathbb{\delta}_4$}};
         \draw[<->,line width=0.15mm,dashed,dash pattern=on 1mm off 0.5mm] (18,7)--(18,8.5) 
        node[midway,right] {$\mathbb{\delta}_5$};
         \draw[<->,line width=0.15mm,dashed,dash pattern=on 1mm off 0.5mm] (18,8.5)--(18,11.5) 
        node[midway,right] {$\mathbb{\delta}_6$};
\end{tikzpicture}
\caption{How to compute the coefficients $c_k$ in the formula for the valuation of $S_4$: for each of the vertices on the geodesic to $a_4\wedge a_5$, count the leaves whose most recent common ancestor is that vertex.\label{fig:ExValuationOfS}}
\end{figure}
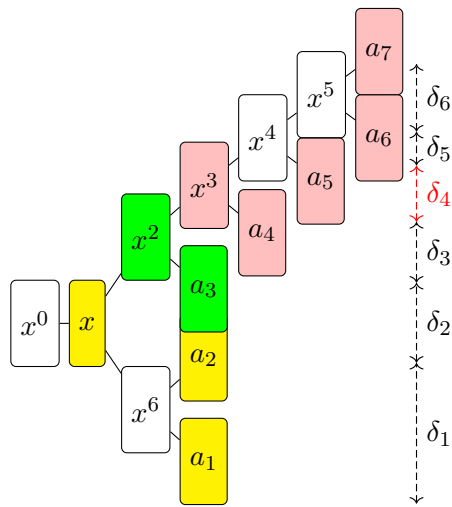

We have $\delta_1=x^6$, $e_1=6$, $\delta_2=x$, $e_2=1,$ $\delta_3=x^2,$ $e_3=2,$ $\delta_4=x^3$, $e_4=3,$ $\delta_5=x^4,$ $e_5=4$, $\delta_6=x^5,$ $e_6=5.$ Let us consider three steps: first we compute the valuation by integration; second we compute the valuation by using formula \ref{ec:ValuationS}; in the end, the third step is to check if the two results coincide.

\ \\
$\bullet$ \emph{Step 1: integration}

We have $$ \mathrm{S}_4(x)=\int_{0}^{1}(\delta_1+\delta_2+
\delta_{3}+s\delta_4)(\delta_2+
\delta_{3}+s\delta_4)(
\delta_{3}+s\delta_4)s\delta_4
 (1-s)\delta_4 ((1-s)\delta_4+\delta_5) ((1-s)\delta_4+\delta_5+\delta_6) (\delta_4\mathrm{d}s)=$$
 $$=\int_{0}^{1}(x^6+x^1+
x^2+sx^3)(x^1+
x^2+sx^3)(
x^2+sx^3)sx^3
 (1-s)x^3 ((1-s)x^3+x^4) ((1-s)x^3+x^4+x^5) (x^3\mathrm{d}s).$$
 We obtain 
$\nu_x( \mathrm{S}_4(x))=\nu_x(\int_{0}^1 x^1 x^1 x^2 s x^3 (1-s)x^3(1-s)x^3 (1-s)x^3 (x^3 \mathrm{d}s)),$ hence 

$\nu_x( \mathrm{S}_4(x))=2\nu_x(x^1)+1\nu_x(x^2)+4\nu_x(x^3)+\nu_x(x^3)=2e_2+1 e_3+4 e_4+e_4=19.$

\ \\
$\bullet$ \emph{Step 2: formula and contact tree}

See Figure \ref{fig:ExValuationOfS}, where we compute the coefficients $c_k$ in the formula for the valuation of $S_4$. The vertex $a_4\wedge a_5$ corresponds to $\delta_4$. Here $a_1\wedge a_2\not\in \mathcal{G}_4$, so $e_1$ does not appear in the sum. For each $G$ on the geodesic $\mathcal{G}_4$ with $ G\leq_{\mathcal{CT}} a_4\wedge a_5$, we count the number of leaves of $\mathcal{CT}$ whose most recent ancestor belonging to $\mathcal{G}_4$ is $G$. Thus $c_{a_2\wedge a_3}=2,$ $c_{a_3\wedge a_4}=1,$ $c_{a_4\wedge a_5}=4.$ 

Thus $\nu_x(\mathrm{S}_4(x))=2\times 1+ 1\times 2+4\times 3+3=19.$

\ \\
$\bullet$ \emph{Step 3: comparison of the two results}

Both methods yield the same result: $\nu_x(\mathrm{S}_4(x))=19.$
\end{example}

\begin{corollary}\label{corol:ValuationS_iAll}
With the hypotheses and notations from Proposition \ref{prop:valuationOfSi}, one can write $$\nu_x(\mathrm{S}_i)=e_i+\sum_{j=1}^i q_{ij}  e_j,$$
where $q_{ij}\in\mathbb{N}$, $
q_{ij}=
\begin{cases}
c_{G}(i), \text{ if } G \text{ is comparable to } a_i\wedge a _{i+1} \text{ and } G\in \mathcal{G}_i,\ G\leq_{\mathcal{CT}} a_i\wedge a_{i+1};\\
0 \text{ else. }

\end{cases}
$

\end{corollary}

The purpose of Corollary \ref{corol:ValuationS_iAll} is to be able to express $\nu_x(\mathrm{S}_i)$ in function of all the $e_j$, for $j=1,i$, even if some of the $e_j$ will have zero coefficient.

Under the notations from Proposition \ref{prop:valuationOfSi}, we conclude this subsection with a lemma and a corollary, which we use later, in Subsection \ref{subsect:realising1var}.

\begin{lemma}\label{lemma:evenIThenUnderAxis}
If  $P_x(y)$ is a monic polynomial such that all its roots $a_i(x)\in\mathcal{A},$ $i=1,\ldots,m+1$ are distinct, then we have $\int_{a_i(x)}^{a_{i+1}(x)}P_x(t)\mathrm{d}t=(-1)^{m+1-i}\ \mathrm{S}_i.$
\end{lemma}

\begin{proof}
 Since $P_x(y)$ is a monic polynomial, we know that $\lim_{y\rightarrow\infty}P_x(y)
=\infty$, and we obtain $$\int_{a_m(x)}^{a_{m+1}(x)}P_x(t)\mathrm{d}t=-\mathrm{S}_m$$ (namely the rightmost surface is always under the horizontal $Oy$-axis). By hypothesis, all the roots of $P_x(y)$ are distinct, thus the areas are alternatively above and under the $Oy$-axis drawn horizontally. One can conclude that a given area $\mathrm{S}_i$ is above $Oy$-axis if and only if $m+1-i$ is even and we have $\int_{a_i(x)}^{a_{i+1}(x)}P_x(t)\mathrm{d}t=(-1)^{m+1-i}\  \mathrm{S}_i.$
\end{proof}

\begin{corollary}\label{cor:Diferenta}
If $\mathrm{S}_i:=\bigg\vert\int_{a_i(x)}^{a_{i+1}(x)}P_x(t)\mathrm{d}t\bigg\vert$, then $\mathrm{S}_i$ is above $Oy$-axis if and only if $m+1-i$ is even.
\end{corollary}

\subsection{Inequalities between areas $\mathrm{S}_i$, read on the contact tree}\label{sect:IneqTree}

Considers the contact tree $\mathcal{CT}$ and two comparable vertices $v_k$ and $v_{\ell}$. The question is whether the order relation between the two vertices gives us an order between the corresponding areas $\mathrm{S}_k$ and $\mathrm{S}_{\ell}.$ 

\begin{proposition}\label{prop:comparable}
 Let $v_k:=a_k\wedge a_{k+1}$ and $v_{\ell}:=a_\ell\wedge a_{\ell+1}$ denote two comparable vertices of the contact tree $\mathcal{CT}$. If $v_k <_{\mathcal{CT}} v_{\ell}$, then we have the polynomial inequality $\mathrm{S}_k\succ_{+}  \mathrm{S}_{\ell}.$ 
\end{proposition}

Before reading the proof of Proposition \ref{prop:comparable}, we suggest the following example. 

\begin{example}\label{ex:LeavesComputation}
For the contact tree in Figure \ref{fig:leaves}, we want to compare the coefficients which appear in the computation of $\nu_x(\mathrm{S}_3)$, i.e. $c_{a_i\wedge a_{i+1}}(3)$ with the coefficients which appear in the computation of $\nu_x(\mathrm{S}_7)$, i.e. $c_{a_i\wedge a_{i+1}}(7)$. We have the inclusion of geodesics $\mathcal{G}_3\subset\mathcal{G}_7.$

Let us count the number of leaves that exit each vertex, for each geodesics.
We find equality in the case of the common vertices of the two geodesics: 

$c_{a_1\wedge a_{2}}(3)=c_{a_1\wedge a_{2}}(7)=1$,

$c_{a_2\wedge a_{3}}(3)=c_{a_2\wedge a_{3}}(7)=1.$

Starting with the vertex $a_3\wedge a_4,$ which is the end on the first geodesic $\mathcal{G}_3$, we obtain different number of leaves which exit the two geodesics:

$c_{a_3\wedge a_{4}}(3)=6$ and $c_{a_3\wedge a_{4}}(7)=1.$

Moreover, since the vertices $a_4\wedge a_{5}$, $a_5\wedge a_{6}$, $a_6\wedge a_{7}$ and $a_7\wedge a_{8}$ are not on the geodesic $\mathcal{G}_3$, it follows that:

$c_{a_4\wedge a_{5}}(3)=c_{a_5\wedge a_{6}}(3)=c_{a_6\wedge a_{7}}(3)=c_{a_7\wedge a_{8}}(3)=0.$ 

However, on the rest of the geodesic $\mathcal{G}_7$, i.e. on $\mathcal{G}_7\setminus\mathcal{G}_3$ we have:

$c_{a_4\wedge a_{5}}(7)=1,$ $c_{a_5\wedge a_{6}}(7)=1;$ 

Since $a_6\wedge a_{7}\notin\mathcal{G}_7$ we have $c_{a_6\wedge a_{7}}(7)=0$.

Finally, we have $c_{a_7\wedge a_{8}}=3,$ wince there are three leaves that exit $\mathcal{G}_7$ at the vertex $a_7\wedge a_{8}.$

\begin{figure}[H]
\begin{center}
  \begin{tikzpicture}[sibling distance=3em,
   somenode/.style = {shape=circle, scale=0.3,fill=black,
    draw, align=center},
     level 1/.style={sibling distance=8em},
  level 2/.style={sibling distance=7em},
  level 3/.style={sibling distance=6em},
  level 4/.style={sibling distance=5em},
  level 5/.style={sibling distance=4em},
  level 6/.style={sibling distance=3em},
  level 7/.style={sibling distance=2em}]

   \node [shape=rectangle,fill=white] {R}
child[grow=right] {node [somenode, label=below:$a_1\wedge a_2$] {}
	child {  node [label=right:$a_{1}$]{} edge from parent[->] }
    child { node[somenode,label=below:$a_{2}\wedge a_3$] {}
    	child { node [label=right:$a_{2}$]{} edge from parent[->]}
      	child { node [somenode,color=red,label=below:$a_{3}\wedge a_4$]{} 
      		child { node [label=right:$ a_3$]{} edge from parent[->]}
      			child { node [somenode,label=below:$a_{4}\wedge a_5$]{} 
      				child { node [label=right:$ a_4$]{} edge from parent[->]}
      				child { node [somenode,label=below:$a_{5}\wedge a_6$]{} 
						child{ node [label=right:$ a_5$]{} edge from parent[->]}
						child{node [somenode,color=red,label=above:$a_{7}\wedge a_8$]{}
							child{ node [somenode,label=below:$a_{6}\wedge a_7$]{} 
									child{ node [label=right:$ a_6$]{} edge from parent[->]}
									child{ node [label=right:$ a_7$]{} edge from parent[->]}
							}
							child{node [label=right:$ a_8$]{} edge from parent[->]{}							
							}
						}      				
      				} 
      			}    		
      	}
   }      
};

\end{tikzpicture}
\caption{Computing the valuation of $\mathrm{S}_3(x)$, respectively of $\mathrm{S}_{7}(x)$.\label{fig:leaves}}
\end{center}
\end{figure}
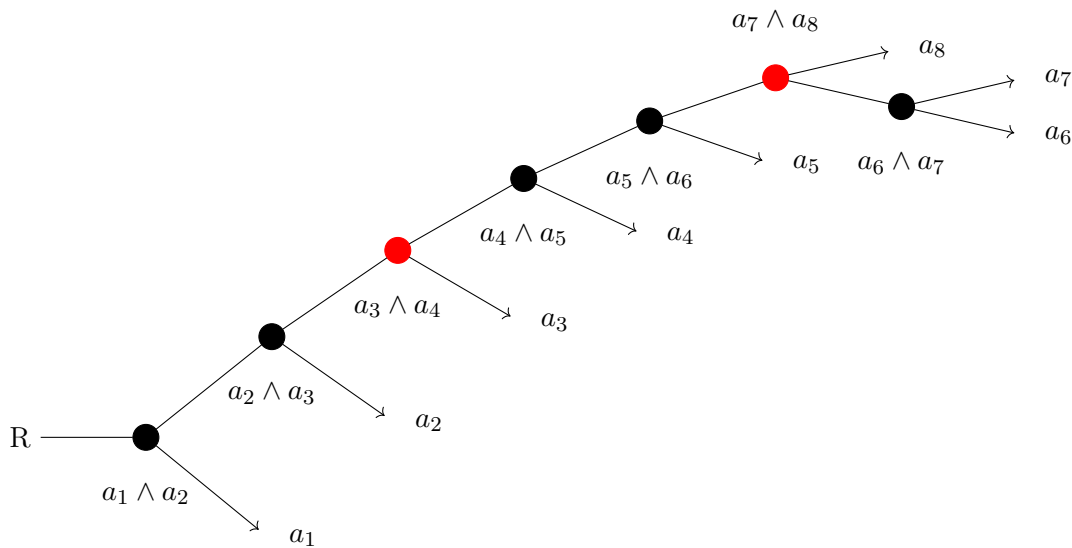
\end{example}

\begin{proof}
Let us prove Proposition \ref{prop:comparable}. By Definition \ref{def:ComparatieInContactTree}, since $v_k<_{\mathcal{CT}} v_{\ell}$,  we obtain that $\mathcal{G}_k\subset\mathcal{G}_\ell.$ The main idea is to compare the number of leaves on each vertex of the geodesic to $v_{\ell}$. Namely, on the common geodesic $\mathcal{G}_k$, the number of leaves that exit the geodesic at a vertex different from $v_k$ is the same both in the computation of $\nu_x(S_k)$ and in the computation of $\nu_x(S_\ell)$: for $\{G\in \mathcal{G}_{\ell}\mid G<_{\mathcal{CT}} a_{k-1}\wedge a_{k}\}$ we have $c_{G}(k)=c_{G}(\ell)$. Furthermore, in the case of $\nu_x(S_k)$, the number of leaves that exit the geodesic at $v_k$ (denoted by $c_{a_k\wedge a_{k+1}}(k)$) is the sum of all the leaves that exit the geodesic $\mathcal{G}_\ell$ computed for $\nu_x(S_\ell)$ i.e. 
\begin{equation}\label{eq:suma}
\sum_{\{G\in \mathcal{G}_{\ell}\mid  a_k\wedge a_{k+1}\leq_{\mathcal{CT}}G\leq_{\mathcal{CT}} a_{\ell}\wedge a_{\ell+1}\}} c_{G}(\ell)=c_{a_k\wedge a_{k+1}}(k).
\end{equation}

Remember Proposition \ref{prop:valuationOfSi}: if we denote by $e_k:=\nu_x(v_k)$, we have $$
\nu_x(\mathrm{S}_k)=e_k+\sum_{\{G\in \mathcal{G}_k\mid G\leq_{\mathcal{CT}} a_k\wedge a_{k+1}\}} c_{G}(k)  \nu_x(G),
$$
where the coefficient $c_{G}(k)\in\mathbb{N}$ represents the number of leaves of the contact tree $\mathcal{CT}(\mathcal{A})$ whose most recent ancestor belonging to $\mathcal{G}_k$ is $G$.

Similarly, if we denote by $e_{\ell}:=\nu_x(v_{\ell})$ we have $$
\nu_x(\mathrm{S}_{\ell})=e_{\ell}+\sum_{\{G\in \mathcal{G}_{\ell}\mid G\leq_{\mathcal{CT}} a_{\ell}\wedge a_{\ell+1}\}} c_{G}(\ell)  \nu_x(G),
$$
where the coefficient $c_{G}(\ell)\in\mathbb{N}$ represents the number of leaves of the contact tree $\mathcal{CT}(\mathcal{A})$, whose most recent ancestor belonging to $\mathcal{G}_{\ell}$ is $G$.

By Definition \ref{def:CT}, since on each geodesic the integers decorating the bifurcation vertices of $\mathcal{CT}(\mathcal{A})$ form a strictly increasing subsequence, we obtain that $e_k<e_{\ell}.$

Thus $$
\nu_x(\mathrm{S}_{\ell})=e_{\ell}+
\sum_{\{G\in \mathcal{G}_{\ell}\mid G\leq_{\mathcal{CT}} a_{\ell}\wedge a_{\ell+1}\}} c_{G}(\ell)  \nu_x(G)
>$$
$$>e_k+
\sum_{\{G\in \mathcal{G}_{\ell}\mid G\leq_{\mathcal{CT}} a_{k-1}\wedge a_{k}\}} c_{G}(\ell)  \nu_x(G)+
\sum_{\{G\in \mathcal{G}_{\ell}\mid  a_k\wedge a_{k+1}\leq_{\mathcal{CT}}G\leq_{\mathcal{CT}} a_{\ell}\wedge a_{\ell+1}\}} c_{G}(\ell)  \nu_x(G).$$

Since for $\{G\in \mathcal{G}_{\ell}\mid G<_{\mathcal{CT}} a_{k-1}\wedge a_{k}\}$ we have $c_{G}(k)=c_{G}(\ell)$ and for $\{G\in \mathcal{G}_{\ell}\mid  a_k\wedge a_{k+1}\leq_{\mathcal{CT}}G\leq_{\mathcal{CT}} a_{\ell}\wedge a_{\ell+1}\}$, we have $\nu_x(G)>e_k$, we obtain
$$\nu_x(\mathrm{S}_{\ell})>e_k+
\sum_{\{G\in \mathcal{G}_{\ell}\mid G<_{\mathcal{CT}} a_{k-1}\wedge a_{k}\}} c_{G}(k)  \nu_x(G)+
\sum_{\{G\in \mathcal{G}_{\ell}\mid  a_k\wedge a_{k+1}\leq_{\mathcal{CT}}G\leq_{\mathcal{CT}} a_{\ell}\wedge a_{\ell+1}\}} c_{G}(\ell)  e_k=$$

$$=e_k+
\sum_{\{G\in \mathcal{G}_{\ell}\mid G<_{\mathcal{CT}} a_{k-1}\wedge a_{k}\}} c_{G}(k)  \nu_x(G)+
e_k\sum_{\{G\in \mathcal{G}_{\ell}\mid  a_k\wedge a_{k+1}\leq_{\mathcal{CT}}G\leq_{\mathcal{CT}} a_{\ell}\wedge a_{\ell+1}\}} c_{G}(\ell).$$

By formula \ref{eq:suma}, $\sum_{\{G\in \mathcal{G}_{\ell}\mid  a_k\wedge a_{k+1}\leq_{\mathcal{CT}}G\leq_{\mathcal{CT}} a_{\ell}\wedge a_{\ell+1}\}} c_{G}(\ell)=c_{a_k\wedge a_{k+1}}(k)$,
which implies that $$\nu_x(\mathrm{S}_{\ell})>e_k+\sum_{\{G\in \mathcal{G}_{\ell}\mid G<_{\mathcal{CT}} a_{k-1}\wedge a_{k}\}} c_{G}(k)  \nu_x(G)+
e_k c_{a_{k}\wedge a_{k+1}}(k)= $$

$$=e_k+\sum_{\{G\in \mathcal{G}_k\mid G\leq_{\mathcal{CT}} a_k\wedge a_{k+1}\}} c_{G}(k)  \nu_x(G)=\nu_x(\mathrm{S}_k).$$

\end{proof}

\begin{remark}\label{remark:eIncreasing}
By Definition \ref{def:CT} and by the initial hypothesis that $a_1\prec_{+} a_2\prec_{+} \ldots\prec_{+} a_{m+1}$, if $V_k <_{\mathcal{CT}} V_{\ell}$, then one has $e_k<e_{\ell}$, which is equivalent to $\delta_k\succ_{+} \delta_{\ell}$. By Proposition \ref{prop:comparable}, $\mathrm{S}_k\succ_{+} \mathrm{S}_{\ell}.$ 
However, if $V_{\ell}$ and $V_k$ are not comparable, then it may happen that either $\mathrm{S}_k\succ_{+} \mathrm{S}_{\ell},$ or $\mathrm{S}_k\prec_{+} \mathrm{S}_{\ell}.$ See Example \ref{Ex:CtrExEtagesSuffitPas}.

\end{remark}
\begin{example}\label{Ex:CtrExEtagesSuffitPas}
Let us study the contact tree given in Figure \ref{Fig:ctrExCtctTree} below and let us compute the valuations of $\mathrm{S}_7$, $\mathrm{S}_1$ and $\mathrm{S}_3$ using Corollary \ref{corol:ValuationS_iAll}. We get  $\nu_x(\mathrm{S}_7)=2\cdot 6+2\cdot 4 +4=24$, $\nu_x(\mathrm{S}_1)=2\cdot 8+1\cdot 6 +3 \cdot 3 + 2\cdot 2 + 8=43$  and  
$\nu_x(\mathrm{S}_3)=2\cdot 2+3 \cdot 6 + 3=25.$ Here the vertices $a_7\wedge a_8$ and $a_3\wedge a_4$ are not comparable. We cannot apply Proposition \ref{prop:comparable}. We have $e_3=3 < e_7=4$ thus $\delta_7\prec_+ \delta_3$, i.e. $S_3 \prec_+ S_7$. On the other hand, for other two non comparable vertices $a_7\wedge a_8$ and $a_1\wedge a_2$ we have $e_7=4<e_1=8$ thus $\delta_1\prec_+ \delta_7$ but we obtain the opposite inequality, namely $S_1 \prec_+ S_7$. 
\begin{figure}[H]
\centering
\includegraphics[scale=0.12,trim={0 4cm 0 2cm},clip]{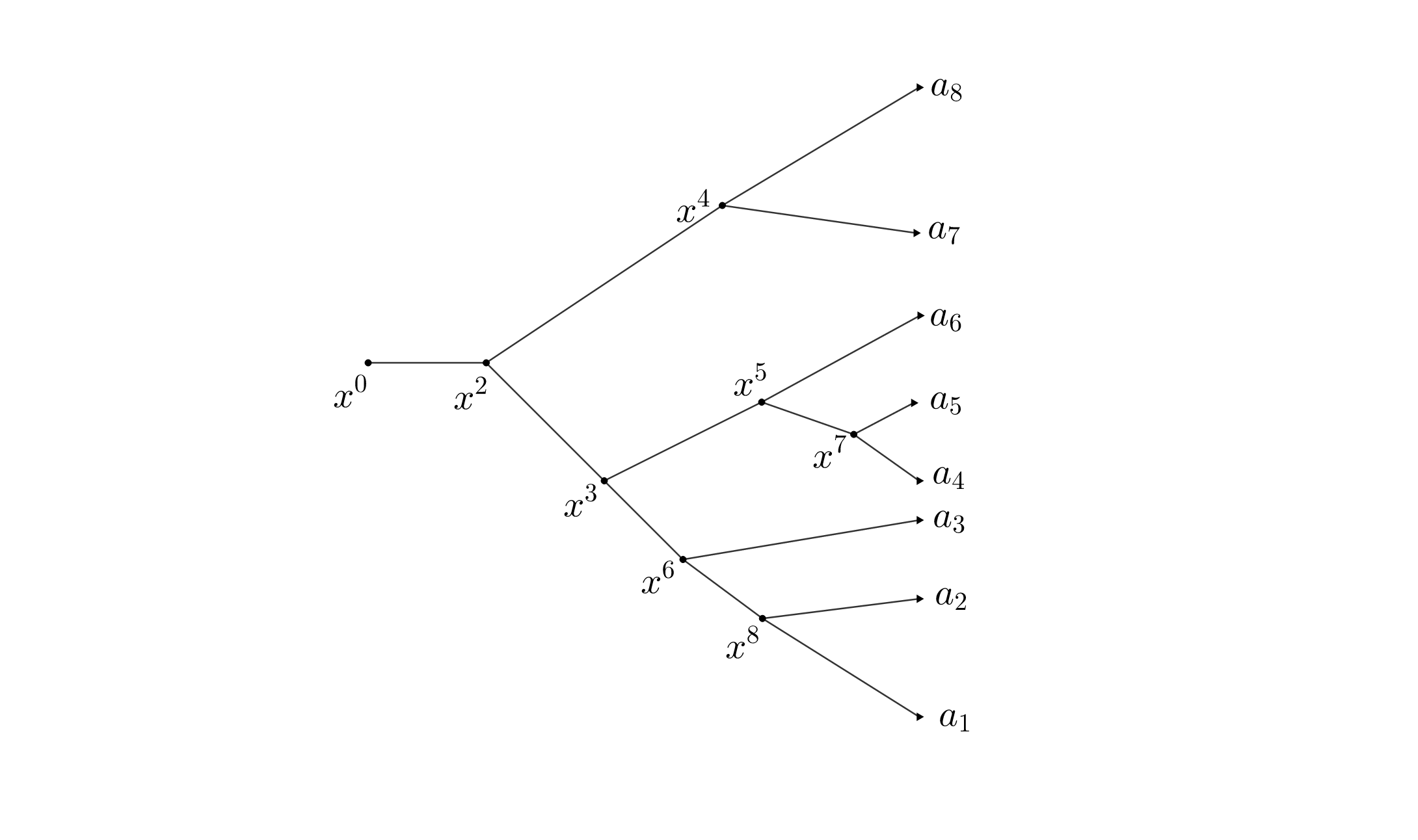}   
\caption{Illustration for Example \ref{Ex:CtrExEtagesSuffitPas}. \label{Fig:ctrExCtctTree}}
\end{figure}

\end{example}

Proposition \ref{prop:valuationOfSumOfAreas}  below represents one of the main arguments in the proof of our main result, namely Theorem \ref{Th:OriginalConstructionOneVariable}.

\begin{proposition}\label{prop:valuationOfSumOfAreas}
Let us consider the  contact tree $\mathcal{CT}(\mathcal{A})$ of the roots $a_i(x)\in\mathcal{A},$ $i=1,\ldots,m+1$ of the polynomial $P_x(y).$ Suppose that the contact tree is complete plane binary. If $i<j$ and $\mathrm{S}_k$ denotes the unique area corresponding to the bifurcation vertex $a_i\wedge a_j$, then 

$$\nu_x\left (\pm\mathrm{S}_i\pm\mathrm{S}_{i+1}
\pm\ldots\pm\mathrm{S}_{j-1})=\nu_x(\mathrm{S}_k\right ).$$
 
\end{proposition}

\begin{proof}
Let us take $\ell \in\{i,\ldots,j-1\}$ such that $\ell \neq k$. By Corollary \ref{coroll:GeneralWedge}, we obtain $v_k<_{\mathcal{CT}}v_\ell.$ Now by Proposition \ref{prop:comparable}, $\nu_x(\mathrm{S}_{\ell})>\nu_x(\mathrm{S}_k).$ Thus $\mathrm{S}_k$ has the minimal valuation in the sum $\pm\mathrm{S}_i\pm\mathrm{S}_{i+1}\pm\ldots\pm\mathrm{S}_{j-1}$. Since the contact tree is binary, this valuation appears only once and we conclude that 
$\nu_x\left (\pm\mathrm{S}_i\pm\mathrm{S}_{i+1}\pm\cdots
\pm\mathrm{S}_{j-1}\right )=\nu_x(\mathrm{S}_k).$
\end{proof}

\section{Separable permutations}
While we know that the existence of Morse polynomials with any given snake was proven by Arnold, the goal of this paper is to give a \emph{constructive} answer, for a special class of snakes, namely the separable ones. 

Let us first define the separable permutations.

\begin{definition}\cite[page 57]{Ki}\label{def:directSkewSum}
Let $\pi:\{1,\ldots,m\}\rightarrow \{1,\ldots,m\}$ and $\sigma:\{1,\ldots,n\}\rightarrow\{1,\ldots,n\}$ be two permutations. Then their \defi{direct sum} $\pi \oplus \sigma$ and their \defi{skew sum} $\pi \ominus \sigma$ are defined as follows:

$\pi \oplus \sigma(i):=\begin{cases}
\pi(i), i=1,\ldots,m;\\
\sigma(i-m)+m,i=m+1,\ldots,m+n.
\end{cases}$

$\pi \ominus \sigma(i):=\begin{cases}
\pi(i)+n, i=1,\ldots,m;\\
\sigma(i-m),i=m+1,\ldots,m+n.
\end{cases}$
\end{definition}

\begin{example}
Let us consider the following two permutations:  $\pi:=\begin{pmatrix}
    1 & 2 & 3& 4 & 5&6\\
    4 &6& 5& 3& 1&2
  \end{pmatrix}
  $ and $\sigma:=\begin{pmatrix}
    1 & 2 & 3\\
    1 &3& 2
  \end{pmatrix}.$
  Thus by Definition \ref{def:directSkewSum}, we obtain:
    \vspace{-.5\baselineskip}
  $$\pi\oplus\sigma=\begin{pmatrix}
    1 & 2 & 3& 4 & 5&6&7&8&9\\
    4 &6& 5& 3& 1&2&7&9&8
  \end{pmatrix}
  $$
  \vspace{-.5\baselineskip}
  $$\pi\ominus\sigma=\begin{pmatrix}
    1 & 2 & 3& 4 & 5&6&7&8&9\\
    7 &9& 8& 6& 4&5&1&3&2
  \end{pmatrix}
  $$
  
  A very useful visual matrix representation of the direct sum (respectively of the skew sum) of $\pi$ and $\sigma$ can be seen in Figure \ref{fig:directSum} (respectively in Figure \ref{fig:skewSum}) below (inspired from \cite[page 4]{Ki}).

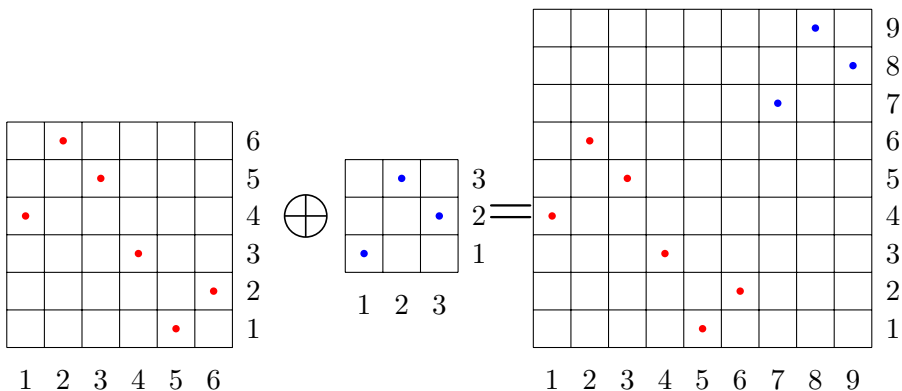
\begin{figure}[H]
\begin{center}
\begin{tikzpicture}[scale=0.5]
\draw [thin, black] (0,0) grid (6,6);
\draw [red] plot [only marks, mark=*,mark options={thick}] coordinates {(0.5,3.5) (1.5,5.5) (2.5,4.5) (3.5,2.5) (4.5,0.5) (5.5,1.5)};

 \node at (0.5,0.3) [below] {$1$};
 \node at (1.5,0.3) [below] {$2$};
 \node at (2.5,0.3) [below] {$3$};
 \node at (3.5,0.3) [below] {$4$};
 \node at (4.5,0.3) [below] {$5$};
 \node at (5.5,0.3) [below] {$6$};
 
  \node at (6.1,0.5) [right] {$1$};
 \node at (6.1,1.5) [right] {$2$};
 \node at (6.1,2.5) [right] {$3$};
 \node at (6.1,3.5) [right] {$4$};
 \node at (6.1,4.5) [right] {$5$};
 \node at (6.1,5.5) [right] {$6$};

\node at (7,3.5) [font=\fontsize{140}{144}\selectfont] { $\oplus$};

\draw [thin, black] (9,2) grid (12,5);
\draw [blue] plot [only marks, mark=*,mark options={thick}] coordinates {(9.5,2.5) (10.5,4.5) (11.5,3.5) };

 \node at (9.5,2.3) [below] {$1$};
 \node at (10.5,2.3) [below] {$2$};
 \node at (11.5,2.3) [below] {$3$};

  \node at (12.1,2.5) [right] {$1$};
 \node at (12.1,3.5) [right] {$2$};
 \node at (12.1,4.5) [right] {$3$};

\node at (12.5,3.5) [font=\fontsize{30}{30}\selectfont] { $=\ \ $};

\draw [thin, black] (14,0) grid (23,9);
\draw [red] plot [only marks, mark=*,mark options={thick}] coordinates {(14.5,3.5) (15.5,5.5) (16.5,4.5) (17.5,2.5) (18.5,0.5) (19.5,1.5) };
\draw [blue] plot [only marks, mark=*,mark options={thick}] coordinates { (20.5,6.5) (21.5,8.5) (22.5,7.5)};

 \node at (14.5,0.3) [below] {$1$};
 \node at (15.5,0.3) [below] {$2$};
 \node at (16.5,0.3) [below] {$3$};
 \node at (17.5,0.3) [below] {$4$};
 \node at (18.5,0.3) [below] {$5$};
 \node at (19.5,0.3) [below] {$6$};
\node at (20.5,0.3) [below] {$7$};
\node at (21.5,0.3) [below] {$8$};
\node at (22.5,0.3) [below] {$9$};

  \node at (23.1,0.5) [right] {$1$};
 \node at (23.1,1.5) [right] {$2$};
 \node at (23.1,2.5) [right] {$3$};
 \node at (23.1,3.5) [right] {$4$};
 \node at (23.1,4.5) [right] {$5$};
 \node at (23.1,5.5) [right] {$6$};
\node at (23.1,6.5) [right] {$7$};
\node at (23.1,7.5) [right] {$8$};
\node at (23.1,8.5) [right] {$9$};

\end{tikzpicture}
\end{center}

\caption{The direct sum $\pi\oplus\sigma$.\label{fig:directSum}}
\end{figure}

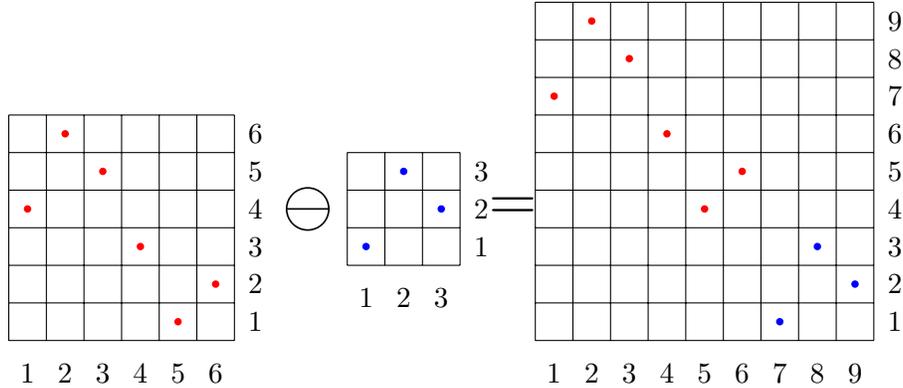
\begin{figure}[H]
\begin{center}
\begin{tikzpicture}[scale=0.5]
\draw [thin, black] (0,0) grid (6,6);
\draw [red] plot [only marks, mark=*,mark options={thick}] coordinates {(0.5,3.5) (1.5,5.5) (2.5,4.5) (3.5,2.5) (4.5,0.5) (5.5,1.5)};

 \node at (0.5,0.3) [below] {$1$};
 \node at (1.5,0.3) [below] {$2$};
 \node at (2.5,0.3) [below] {$3$};
 \node at (3.5,0.3) [below] {$4$};
 \node at (4.5,0.3) [below] {$5$};
 \node at (5.5,0.3) [below] {$6$};
 
  \node at (6.1,0.5) [right] {$1$};
 \node at (6.1,1.5) [right] {$2$};
 \node at (6.1,2.5) [right] {$3$};
 \node at (6.1,3.5) [right] {$4$};
 \node at (6.1,4.5) [right] {$5$};
 \node at (6.1,5.5) [right] {$6$};

\node at (7,3.5) [font=\fontsize{140}{144}\selectfont] { $\ominus$};

\draw [thin, black] (9,2) grid (12,5);
\draw [blue] plot [only marks, mark=*,mark options={thick}] coordinates {(9.5,2.5) (10.5,4.5) (11.5,3.5) };

 \node at (9.5,2.3) [below] {$1$};
 \node at (10.5,2.3) [below] {$2$};
 \node at (11.5,2.3) [below] {$3$};

  \node at (12.1,2.5) [right] {$1$};
 \node at (12.1,3.5) [right] {$2$};
 \node at (12.1,4.5) [right] {$3$};

\node at (12.5,3.5) [font=\fontsize{30}{30}\selectfont] { $=\ \ $};

\draw [thin, black] (14,0) grid (23,9);
\draw [red] plot [only marks, mark=*,mark options={thick}] coordinates {(14.5,6.5) (15.5,8.5) (16.5,7.5) (17.5,5.5) (18.5,3.5) (19.5,4.5) };
\draw [blue] plot [only marks, mark=*,mark options={thick}] coordinates { (20.5,0.5) (21.5,2.5) (22.5,1.5)};

 \node at (14.5,0.3) [below] {$1$};
 \node at (15.5,0.3) [below] {$2$};
 \node at (16.5,0.3) [below] {$3$};
 \node at (17.5,0.3) [below] {$4$};
 \node at (18.5,0.3) [below] {$5$};
 \node at (19.5,0.3) [below] {$6$};
\node at (20.5,0.3) [below] {$7$};
\node at (21.5,0.3) [below] {$8$};
\node at (22.5,0.3) [below] {$9$};

  \node at (23.1,0.5) [right] {$1$};
 \node at (23.1,1.5) [right] {$2$};
 \node at (23.1,2.5) [right] {$3$};
 \node at (23.1,3.5) [right] {$4$};
 \node at (23.1,4.5) [right] {$5$};
 \node at (23.1,5.5) [right] {$6$};
\node at (23.1,6.5) [right] {$7$};
\node at (23.1,7.5) [right] {$8$};
\node at (23.1,8.5) [right] {$9$};

\end{tikzpicture}
\end{center}

\caption{The skew sum $\pi\ominus\sigma$.\label{fig:skewSum}}
\end{figure}
\end{example}

The notion of separable permutation was introduced in \cite{BBL}.

\begin{definition}\cite[page 57]{Ki}\label{def:separablePerm}
A \defi{separable permutation} is a permutation that is obtained by applying several times the $\oplus$ and $\ominus$ operations to the identity permutation of a single element, denoted by $\boxdot:=\begin{pmatrix}
    1\\
   1
  \end{pmatrix}$.
\end{definition}

See Example \ref{ex:DecompositionsSEp}.

\begin{definition}\label{def:decomposition}
Given  a separable permutation $\sigma$, \defi{a decomposition of} $\sigma$ consists of a sequence of operations $\oplus$ and $\ominus$ applied to the identity permutation $\boxdot$, such that the result gives us $\sigma.$
\end{definition}

\begin{remark}\label{remark:associativeominusoplus}
The $\oplus$ operation is (individually) associative and the $\ominus$ operation is (individually) associative (see \cite[page 2]{AHP}). Therefore, a separable permutation may have several distinct decompositions.
\end{remark}

For a new characterisation of the set of separable permutations in terms of graphs of univariate polynomials in the neighbourhood of a common zero and in terms of sub-patterns, the reader should refer to \' Etienne Ghys's recent book \cite[page 27]{Gh1}.

\begin{definition}\cite[page 5]{Ki} 
An \defi{interval} is a non-empty sequence of contiguous integers. 
\end{definition}

\begin{example*}
For instance, a permutation is an interval.
\end{example*}

Recall Definition \ref{Def:rootedCompleteTree} of an end-rooted complete plane binary tree. The following definition of a binary separating tree associated to a separable permutation follows the one given in \cite[page 280]{BBL}. 

\begin{definition}\label{def:binSepTreeAssoc}
If $\sigma$ is a separable permutation of $\{1,2,\ldots,m+1\}$, then \defi{a binary separating tree associated to the separable permutation} $\sigma$ is a complete plane binary tree $T$ such that:
\begin{enumerate}
\item its root is at the top;

\item its leaves are decorated with $\sigma(1)$, $\sigma(2)$,\ldots, $\sigma(m+1)$ in this order from left to right;
\item for every internal vertex $v$, the leaves (seen as a non-ordred set of numbers, for instance (2 3 1)) of the subtree of $v$ form an interval, which we call \textbf{the interval of the node} $v$;
\item let us denote by $v_{left}$ and by $v_{right}$ the left child and the right child of the node $v$, respectively; the internal vertices of $T$ are decorated with $\oplus$ or $\ominus$  as follows:

-if all the numbers of the interval of $v_{left}$ are smaller than all the numbers of the interval of $v_{right}$, then $v$ is \textbf{a positive vertex} and is decorated with the $\oplus$ sign;

-if all the numbers of the interval of $v_{left}$ are bigger than all the numbers of the interval of $v_{right}$, then $v$ is \textbf{a negative vertex} and is decorated with the $\ominus$ sign.
\end{enumerate}
To obtain a separating tree, one should follow step by step a decomposition in direct sums and skew sums of a separable permutation. 
\end{definition}

\begin{remark*}
To a separable permutation one may associate several binary separating trees: see Example \ref{ex:DecompositionsSEp}. This is due to the associativity mentioned in Remark \ref{remark:associativeominusoplus}.  
\end{remark*}

\begin{example}\label{ex:DecompositionsSEp}
An example of a separable permutation is
$\sigma:=\begin{pmatrix}
    1 &2& 3& 4& 5& 6& 7 \\
   6&7&4&5&1&3&2
  \end{pmatrix}$. Two possible ways in which $\sigma$ can be decomposed by repeatedly applying the $\oplus$ and $\ominus$ operations are presented in Figure \ref{fig:descomp1} and in Figure \ref{fig:descomp2} . The algorithm from Definition \ref{def:binSepTreeAssoc} gives us two binary separating trees of $\sigma$. The non-unicity of the decompositions of a separable permutation is due to the associativity of the operation $\oplus$ and of the associativity of the operation $\ominus$.

\begin{figure}[H]

\centering
\includegraphics[scale=0.3]{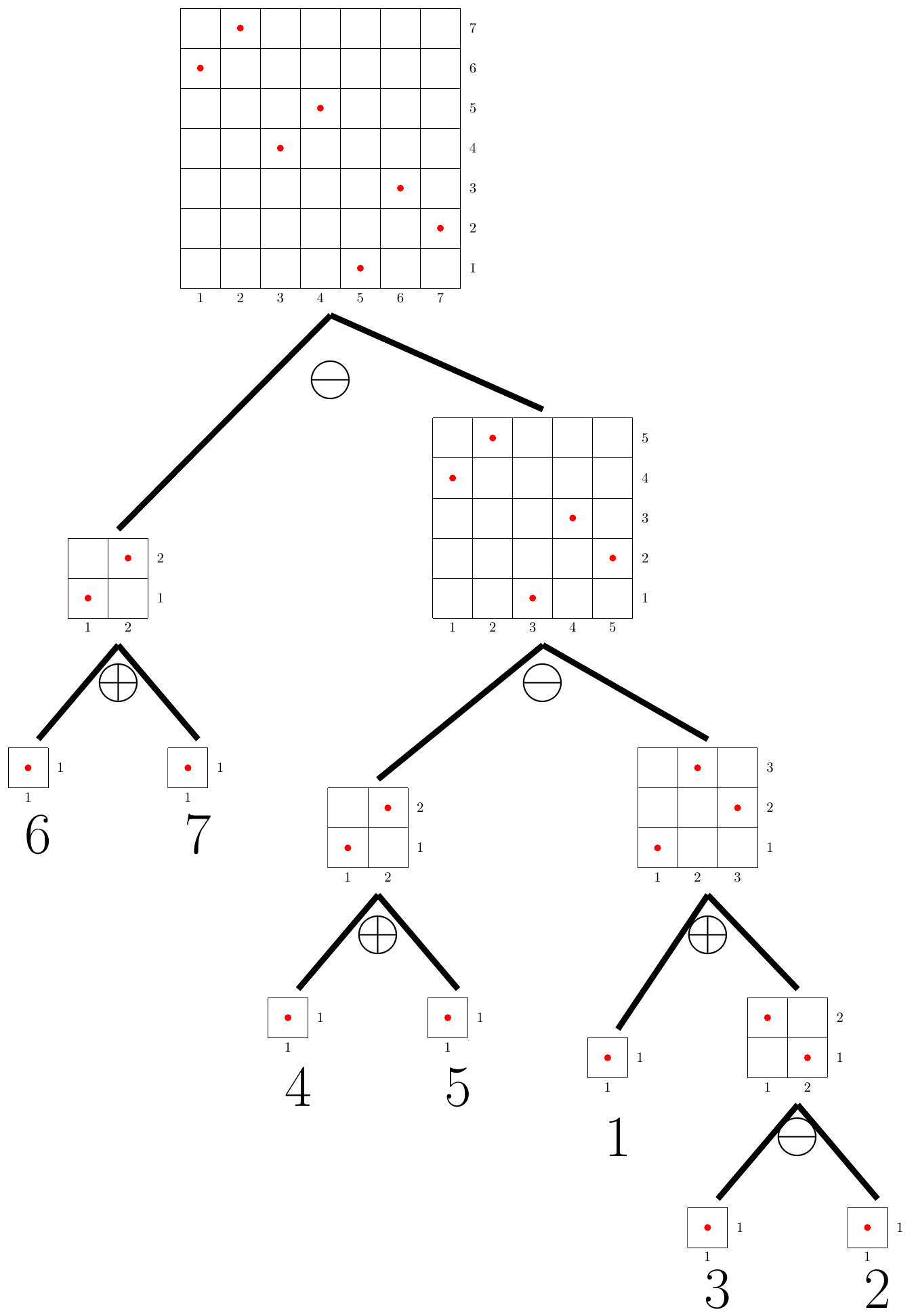} 
\caption{One possible decomposition: $\sigma=(\boxdot\oplus \boxdot)\ominus ((\boxdot\oplus \boxdot)\ominus (\boxdot\oplus(\boxdot\ominus \boxdot))).$  \label{fig:descomp1}}
\end{figure}

\begin{figure}[H]
\centering
\includegraphics[scale=0.3]{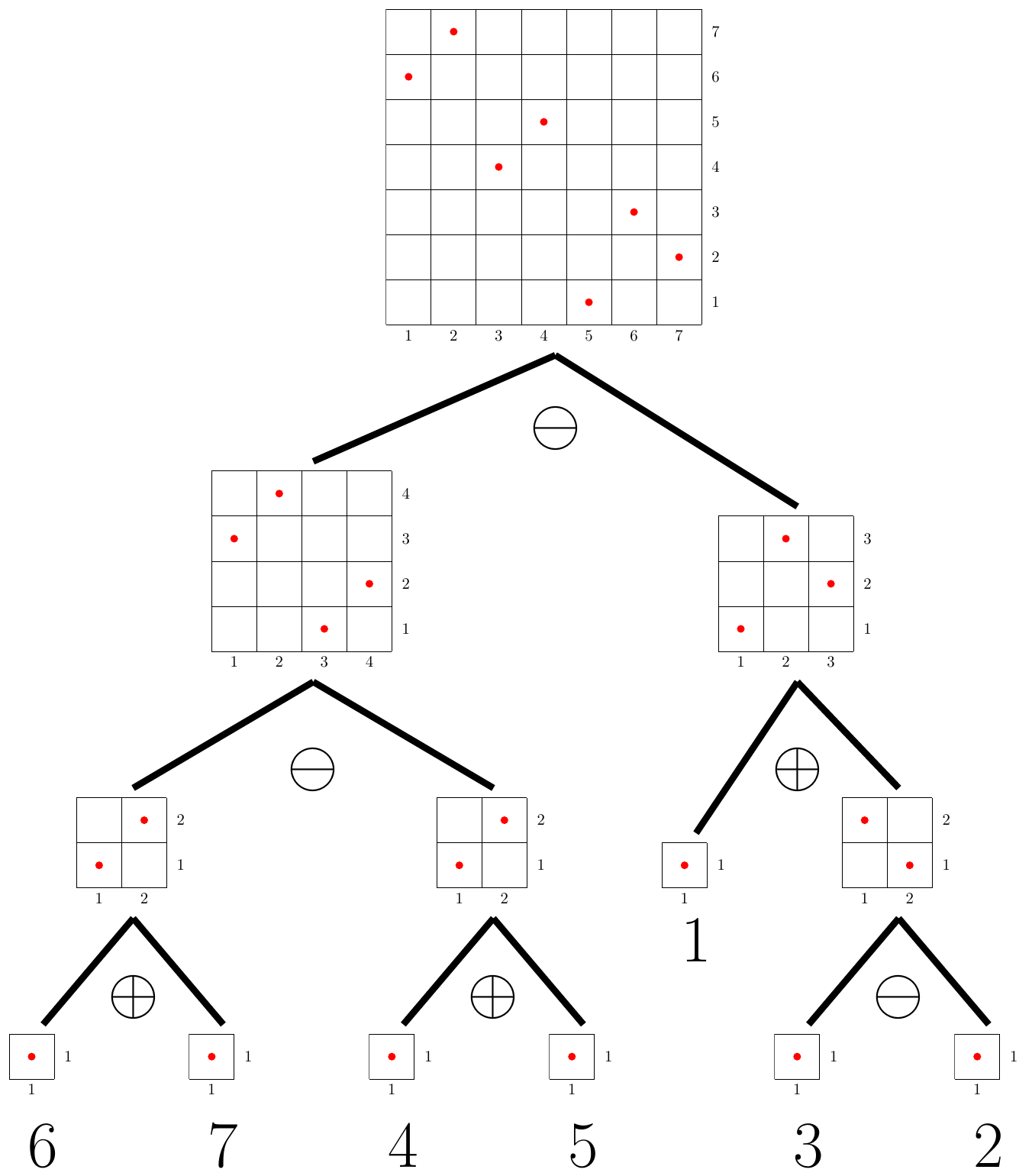} 
\caption{Another possible decomposition: $\sigma=((\boxdot\oplus\boxdot)\ominus (\boxdot \oplus \boxdot))\ominus (\boxdot\oplus(\boxdot\ominus \boxdot)).$  \label{fig:descomp2}}
\end{figure}
\end{example}

\begin{remark*}\cite[page 3]{BoRo}
If one also considers non-binary separating trees, then each separable permutation possesses a unique contracted tree, obtained from any of its binary separating trees by contracting all edges between vertices decorated with the same sign (see Figure \ref{fig:descomp3}). This is again due to the associativity properties of $\oplus$, respectively $\ominus$, operations. We shall not go into further details, since our interest are the binary trees associated to separable permutations. See \cite[page 4]{AHP}.
\end{remark*}
\begin{example}
Let us consider the same separable permutation $\sigma:=\begin{pmatrix}
    1 &2& 3& 4& 5& 6& 7 \\
   6&7&4&5&1&3&2
  \end{pmatrix}$ from Example \ref{ex:DecompositionsSEp}. Thanks to the associativity of the $\ominus$ operation, we can contract any of its two binary separating trees  (see Figure \ref{fig:descomp1} and Figure \ref{fig:descomp2}) in a non-binary contracted tree (in Figure \ref{fig:descomp3}). This can be repeated every time a parent and its child have the same sign (here, $\ominus$).

\begin{figure}[H]
\centering
\includegraphics[scale=0.3]{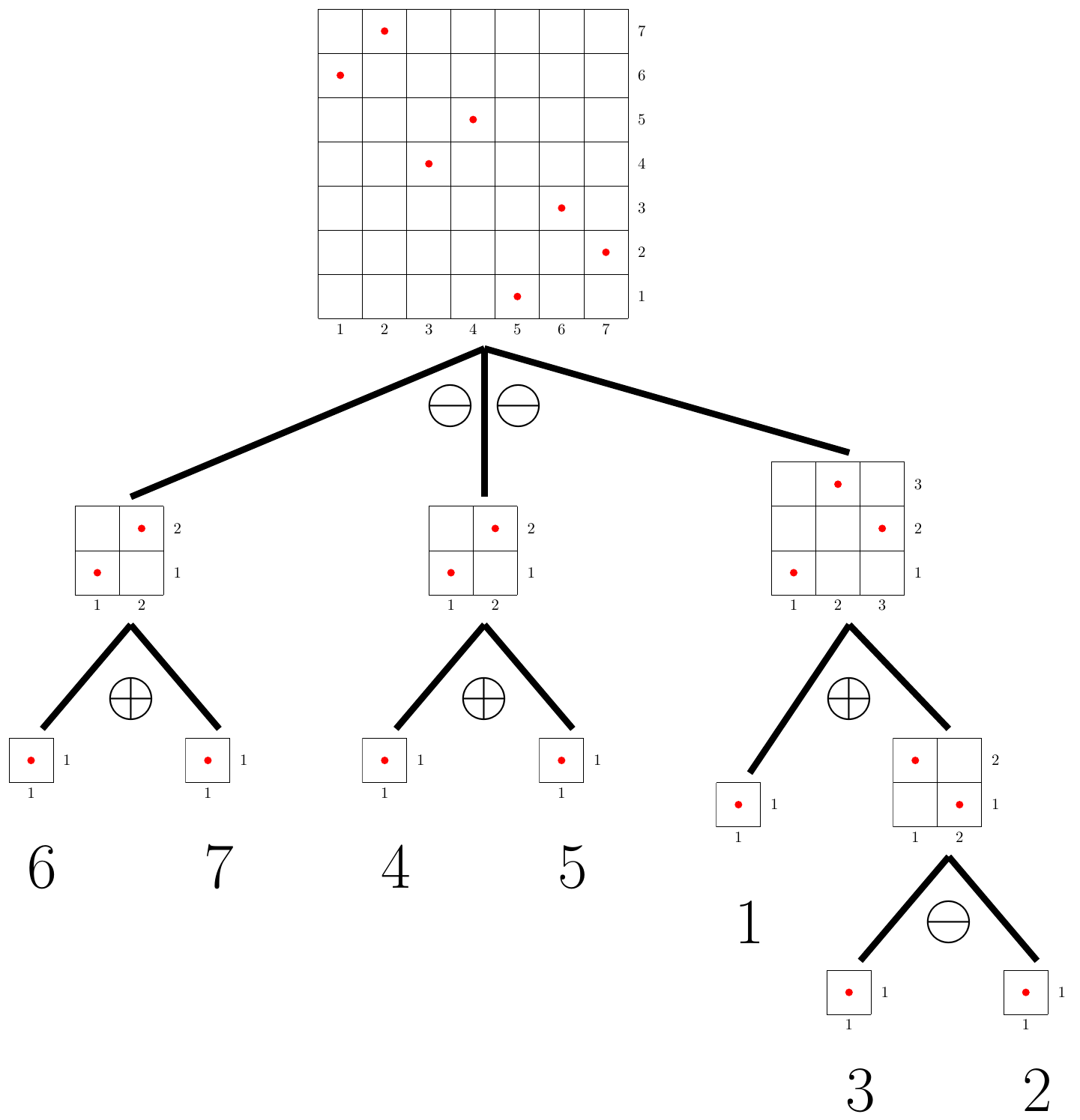} 
\caption{The associativity of $\ominus$ gives us $\sigma=(\boxdot\oplus\boxdot)\ominus (\boxdot \oplus \boxdot)\ominus (\boxdot\oplus(\boxdot\ominus \boxdot))$ and the unique contracted tree associated to a separable permutation from any of its binary separating trees. \label{fig:descomp3}}
\end{figure}
\end{example}

\begin{remark}\label{remark:recursivelyDefinedSets}
Below we provide a recursive definition of the set of complete plane binary trees (inspired from \cite[page 308]{Kn}). The first step is the basic one: we specify an initial collection of elements of the set. The second step is the recursive one: we give a rule to form new elements from those already known to be in the set.

The advantage of using the recursion for this definition lies in the fact that we will use structural induction (see \cite{Shaf}) for the proof of Proposition \ref{prop:BijectionDecomposAndTress}.

The set of \defi{complete plane binary trees} is defined recursively as follows:
\begin{enumerate}
\item the tree consisting of a single vertex is a complete plane binary tree whose root is itself.
\item if $\mathcal{T}_1$ and $\mathcal{T}_2$ are disjoint complete plane binary trees whose roots are $\mathcal{R}_1$ and $\mathcal{R}_2$ respectively, let us denote by $\mathcal{T}$ the tree consisting of a root $\mathcal{R}$ with edges connecting $\mathcal{R}$ to the roots $\mathcal{R}_1$ and $\mathcal{R}_2$ such that $\mathcal{T}_1$ is the left subtree and $\mathcal{T}_2$ is the right subtree. Then $\mathcal{T}$ is also a complete plane binary tree.
\end{enumerate} 
\end{remark}

By Definition \ref{def:separablePerm}, the separable permutations have a recursive structure. We give now a recursive definition (see Remark \ref{remark:recursivelyDefinedSets}) of these binary decompositions that are obtained by applying several times the $\oplus$ and $\ominus$ operations. 

\begin{definition}\label{def:binaryDecomp}
\defi{The set of binary decompositions} is defined recursively as follows:
\begin{enumerate}
\item the trivial permutation $\boxdot:=\begin{pmatrix}
    1\\
   1
  \end{pmatrix}$ is itself a binary decomposition; 
\item if $d_1$ and $d_2$ are two binary decompositions, then the direct sum $d_1 \oplus d_2$ and the skew sum $d_1 \ominus d_2$ are both binary decompositions.

\end{enumerate}
\end{definition}

The above recursive definition will play an important role in establishing a bijection between the set of all the binary decompositions of a separable permutation $\sigma$ and the set of all the binary separating trees associated to $\sigma$ (recall the recursive Definition \ref{DefcompleteTree} and Definition \ref{def:binSepTreeAssoc}), as Proposition \ref{prop:BijectionDecomposAndTress} shows.

\begin{notation}
If $T_1$ and $T_2$ are two complete plane binary trees, then by $T_1\oplus T_2$ (respectively $T_1\ominus T_2$) we denote the new complete plane binary tree obtained by creating a new vertex decorated with the sign $\oplus$ (respectively $\ominus$) such that its left subtree is $T_1$ and its right subtree is $T_2$.
\end{notation}

A reformulation of Definition \ref{def:separablePerm} is the following:
\begin{proposition}\cite[page 280]{BBL}\label{Prop:charactOfSeparable}
A permutation $\sigma$ is separable \emph{if and only if} there exists a \textbf{binary} tree that is a separating tree  associated to $\sigma$.
\end{proposition}

\begin{proposition}\label{prop:BijectionDecomposAndTress}
There is a bijection between the set of all the binary decompositions of a separable permutation $\sigma$ and the set of all the binary separating trees associated to $\sigma$.
\end{proposition}

\begin{proof}
We give an inductive proof that follows the structure of the recursively defined sets (see Definition \ref{DefcompleteTree} and Definition \ref{def:binaryDecomp}).  
The proof has two parts: (a), where we show that the proposition holds for all the minimal structures of the set, and (b) where we prove that if it holds for the immediate substructures of a certain structure $\mathcal{S}$ then it must hold for $\mathcal{S}$ too.

\begin{enumerate}
\item First let us show by structural induction that to every binary decomposition one can associate a unique binary separating tree.

- to the trivial permutation one can associate uniquely the tree with one vertex;

- if $d_1$ and $d_2$ are two binary decompositions which have each a uniquely associated binary separating tree $T_1$ and $T_2$ respectively, then to $d_1 \oplus d_2$ (respectively $d_1 \ominus d_2$) we shall associate the unique binary separating tree $T_1 \oplus T_2$
 (respectively $T_1 \ominus T_2$).
\item A similar structural recursive proof can be given to show that to each binary separating tree one can associate a unique binary decomposition.
\end{enumerate}

\end{proof}

\begin{remark*}
Given a complete plane binary tree such that all its vertices, except from the leaves are decorated with $\oplus$ or $\ominus$ sign, we can construct the unique separable permutation corresponding to it, by using Proposition \ref{prop:BijectionDecomposAndTress} to obtain the decomposition of the permutation.

In other words, one can see any binary decomposition of a separable permutation $\sigma$ as a complete plane binary tree which is in fact one of the binary separating trees of $\sigma$. In addition, if the binary decomposition has $m$ signs, then  the tree has $m$ internal nodes and $m+1$ leaves.
\end{remark*}

\begin{example}
An example which illustrates the bijection from Proposition \ref{prop:BijectionDecomposAndTress}
is shown in Figure \ref{fig:descomp1} and Figure \ref{fig:descomp2}.
\end{example}

We have now a Corollary of Proposition \ref{Prop:charactOfSeparable}, as follows:
\begin{corollary}
A permutation $\sigma$ is separable \emph{if and only if} $\sigma$ has a \textbf{binary} separating decomposition.
\end{corollary}

We conclude by emphasizing the fact that there is a bijective correspondence between the signs in a binary decomposition of $\sigma$ and the internal vertices of the corresponding binary separating trees associated to $\sigma$. In addition, the internal vertices are not only decorated with the corresponding signs $\oplus$ or $\ominus$, but they also correspond to a certain matrix as one can see in the following definition.

\begin{definition}\label{def:wedgeInSeparatingTRee}
Let us consider a separable permutation $\sigma$ and one of its binary separating trees, say $T(\sigma)$. Let us suppose that  $i<j$. We denote by $\sigma(i)\wedge \sigma(j)$ the internal vertex of $T(\sigma)$ where the geodesics from the root to the leaves $\sigma(i)$ and $\sigma(j)$ separate. 
\end{definition}

\begin{definition}
Let us consider a separable permutation $\sigma$. Let us denote by $\mathrm{mat}(\sigma(i)\wedge \sigma(j))$ the minimal submatrix of $\sigma$ that contains $\sigma(i)$ and $\sigma(j)$.
\end{definition}

In order to understand better the notion of minimal submatrix of a permutation, the reader is invited to read Example \ref{ex:submatrix}.

\begin{example}\label{ex:submatrix}
If $\sigma:=\begin{pmatrix}
    1 &2& 3& 4& 5& 6& 7 \\
   6&7&4&5&1&3&2
  \end{pmatrix}$, see Figure \ref{fig:wedgePermutation} (in red) the matrix $\mathrm{mat}(\sigma(1)\wedge \sigma(4))=\begin{pmatrix}
    1 &2& 3& 4 \\
   3&4&1&2
  \end{pmatrix}$, which represents the minimal submatrix of $\sigma$ that contains both $\sigma(1)=6$ and $\sigma(4)=5$. In addition, we have $\mathrm{sign}(\sigma(1)\wedge \sigma(4))=\ominus$. Note that $\sigma(1)>\sigma(4),$ i.e. $6>5.$
  
\begin{figure}[H]
\centering
\includegraphics[scale=0.45]{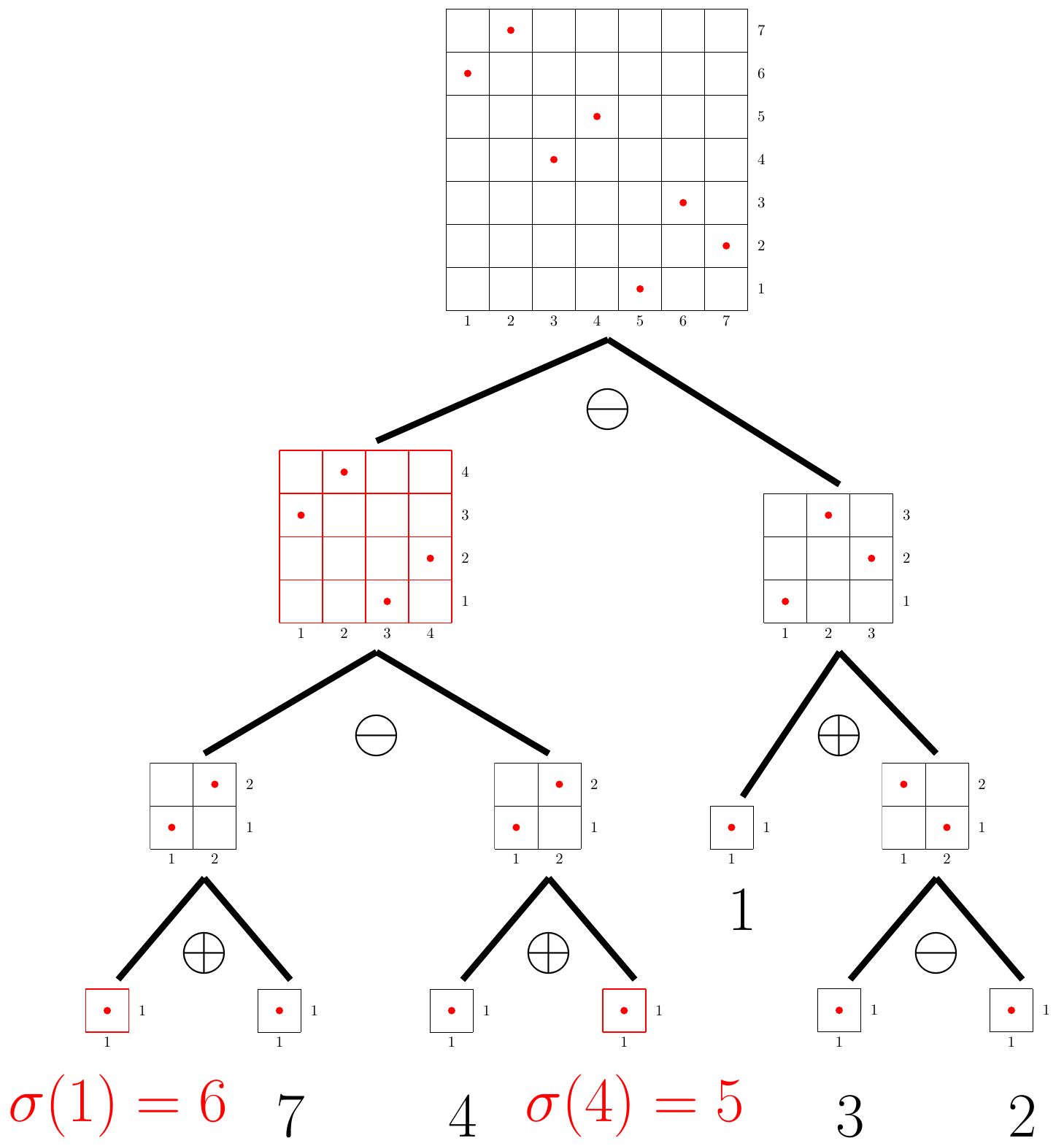} 
\caption{The matrix (in red) of $\sigma(1)\wedge \sigma(4)$ (in red) and the sign $\ominus$ of its corresponding vertex in the separating tree of $\sigma$.\label{fig:wedgePermutation}}
\end{figure}  
\end{example}

\begin{proposition}\label{prop:LegatPermutareSemn}
Let us consider a separable permutation $\sigma$ and one of its binary separating trees, say $T(\sigma)$. Let us suppose that  $i<j$. If one denotes by $\sigma(i)\wedge \sigma(j)$ the internal vertex of $T(\sigma)$ where the geodesics from the root to the leaves $\sigma(i)$ and $\sigma(j)$ separate, then $$\sigma(i)<\sigma(j) \Leftrightarrow \mathrm{sign}(\sigma(i)\wedge \sigma(j))=\oplus$$ and $$\sigma(i)>\sigma(j) \Leftrightarrow \mathrm{sign}(\sigma(i)\wedge \sigma(j))=\ominus.$$ 
\end{proposition}

\begin{proof}
We have $\mathrm{sign}(\sigma(i)\wedge \sigma(j))=\oplus$ (respectively $\mathrm{sign}(\sigma(i)\wedge \sigma(j))=\ominus$) if and only if we can decompose the submatrix $\mathrm{mat}(\sigma(i)\wedge \sigma(j))=m_1 \oplus m_2$ (respectively $\mathrm{mat}(\sigma(i)\wedge \sigma(j))=m_1 \ominus m_2$) such that $m_1$ contains $\sigma(i)$ and $m_2$ contains $\sigma(j)$. Now by Definition \ref{def:directSkewSum}, $\mathrm{mat}(\sigma(i)\wedge \sigma(j))=m_1 \oplus m_2$ if and only if $\sigma(i)<\sigma(j)$ (respectively $\mathrm{mat}(\sigma(i)\wedge \sigma(j))=m_1 \ominus m_2$ if and only if $\sigma(i)>\sigma(j)$).
\end{proof}

\section{Realising any given separable snake as an Arnold snake}\label{subsect:realising1var}

We are now ready to prove the main result of this paper:
\begin{leftbar}
\begin{theorem}\label{Th:OriginalConstructionOneVariable}
Consider $m\in\mathbb{N}$ and fix a separable $(m+1)$-snake $\sigma:\{1,2,\ldots,m+1\}\rightarrow \{1,2,\ldots,m+1\}$ such that $\sigma(m)>\sigma(m+1)$. Choose the polynomials $a_i(x)\in\mathbb{R}[x]$ such that their contact tree is one of the binary separating trees of $\sigma$ and construct a polynomial $Q_x(y)\in\mathbb{R}[x][y]$,
$$Q_x(y):=\int_{0}^y\prod_{i=1}^n\big (t-a_i(x)\big )\mathrm{d}t.$$ 
Then for sufficiently small $x>0$, $Q_x(y)$ is $(m+2)$-Morse (see Definition \ref{def:MorsePolyn}) and the Arnold snake associated to $Q_x(y)$ is the given snake $\sigma$.
\end{theorem}
\end{leftbar}

Note that effective constructions of (counter-)examples of families of (multivariate) polynomials with certain properties is a subject of interest in the study of the geometry and topology of real algebraic varieties: see for instance results from Bodin (\cite{Bod2}), Brugall\'{e} \cite{Bru1}, Itenberg (\cite{II2}), Sturmfels (\cite{Stu1}), Viro (\cite{Vi2}).

Before the proof of Theorem \ref{Th:OriginalConstructionOneVariable}, let us prove some useful lemmas. 

\begin{lemma}\label{lemma:AlternDecompos}
Let $\sigma:\{1,2,\ldots,m+1\}\rightarrow \{1,2,\ldots,m+1\}$ be a separable permutation. A binary decomposition of $\sigma$ has alternating signs $\oplus$ and $\ominus$ if and only if $\sigma$ is an $(m+1)$-snake (see Definition \ref{def:snake}).
\end{lemma}

\begin{proof}
By Definition \ref{def:directSkewSum} and  Definition \ref{def:separablePerm}, if there are two signs $\oplus$ (respectively two signs $\ominus$) that appear consecutively in the binary decomposition, then $\sigma$ is not a snake. Reversely, if $\sigma$  is not a snake, then it has three leaves such that $\sigma(i)<\sigma(i+1)<\sigma(i+2)$ (respectively $\sigma(i)>\sigma(i+1)>\sigma(i+2)$) and this implies the existence of two signs $\oplus$ (respectively two signs $\ominus$) that appear consecutively in the binary decomposition.
\end{proof}

\begin{corollary}\label{cor:LastSignMinus}
Let $\sigma:\{1,2,\ldots,m+1\}\rightarrow \{1,2,\ldots,m+1\}$ be a separable snake such that $\sigma(m)>\sigma(m+1).$ Then any  binary decomposition of $\sigma$ has alternating signs $\oplus$ and $\ominus$ such that the rightmost sign is $\ominus$.
\end{corollary}

\begin{lemma}\label{lemma:RealisingCtctTree}
Let $K_+$ be an end-rooted complete plane binary tree, which has $m+1$ leaves. Then there exist $m+1$ polynomials with real coefficients $a_i(x)\in \mathbb{R}[x]$, for $i=1,\ldots,m+1$, such that $a_i(0)=0$ for $i=1,\ldots,m+1$ and
$K_+=\mathcal{CT}(a_1(x),\ldots,a_{m+1}(x)),$ where $\mathcal{CT}(a_1(x),\ldots,a_{m+1}(x))$ represents the contact tree of the polynomials $a_i.$
\end{lemma}

The result of Lemma \ref{lemma:RealisingCtctTree} is already mentioned in \cite[page 29]{Gh1}. Let us provide a constructive proof.

\begin{proof}
Construct the polynomials $a_i(x)\in \mathbb{R}[x]$ as follows: decorate each internal vertex, i.e. bifurcation vertex, with a positive integer number such that the sequence of integer numbers is strictly increasing on each geodesic from the root towards any leaf. The root is decorated with $0$.
In other words: by taking into account the fact that on each geodesic starting from the root the valuations of the monomials in $x$ are increasing, we can assign to each vertex a valuation, say $\nu(v):=p\in\mathbb{N}$ (the root is assigned the $0$ valuation); since $K_+$ is an end-rooted complete plane binary tree by hypothesis, each internal vertex $v$ of $K_+$ has exactly two children. Each of the two children of $v$ is connected to $v$ by an edge. Since $K_+$ is embedded in the real plane, one can decide which is the first child and which is the second child of $v$, by the induced orientation of $\mathbb{R}^2.$ Now let us label the edge connecting $v$ to its first child with the coefficient $c_{first}(v):=0$ (obtaining thus the monomial $0\times x^p$). Similarly, let us label the edge connecting $v$ to its second child with the coefficient $c_{second}(v):=1$ (obtaining thus the monomial $1\times x^p$). The unique edge starting from the root will have coefficient $0$.
 Now for each leaf $a_i(x)$ we add the monomials on the geodesic $\mathcal{G}_i$ from the root to  $a_i(x)$, thus obtaining $a_i(x)=\sum_{v\in\mathcal{G}_i}c_
 {*}(v)x^{\nu(v)}$, where $*$ is either \emph{first} or \emph{second}, depending on the geodesic we are following. 
\end{proof}

\begin{remark*}
The condition $a_i(0)=0$ for $i=1,\ldots,m+1$ is realisable since by hypothesis the tree $K_+$ is an \textbf{end-rooted} complete plane binary tree. 
\end{remark*}

\begin{example}\label{ex:ConstructionOfAi}
Given the end-rooted complete plane binary tree $K_+$ with $5$ leaves like in Figure \ref{fig:exCtcTreeRealised}, one can construct the following polynomials $a_i(x)$, $i=1,\ldots,5$ such that they realise $K_+$ as their contact tree: $a_1(x):=0,$ $a_2(x):=x^3,$ $a_3(x)=:x^2,$ $a_4(x):=x^2+x^5$ and $a_5(x):=x^2+x^5+x^6.$
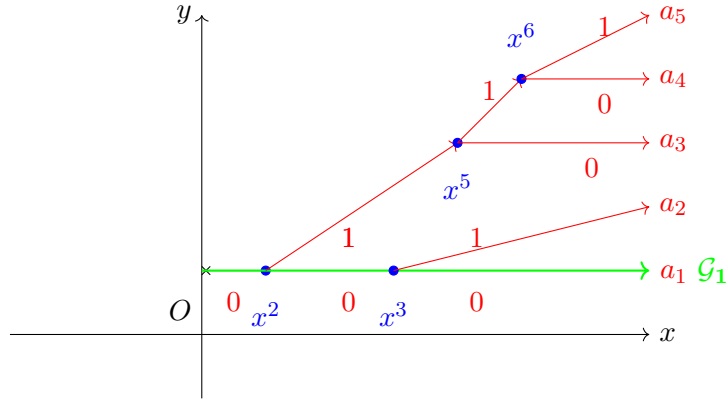
\begin{figure}[H]
\centering
\tikzset{cross/.style={cross out, draw=black, fill=none, minimum size=2*(#1-\pgflinewidth), inner sep=0pt, outer sep=0pt}, cross/.default={2pt}}
\begin{tikzpicture}[scale=0.85]
\tikzstyle{vertex} = [draw,blue]
\tikzstyle{edge} = [draw,red,thick]
\draw[->] (-3,-6) -- (7,-6) node[right] {$x$};
\draw[->] (0,-7) -- (0,-1) node[left] {$y$};
\node at (0,-6.3) [above left] {$O$};
\draw[ color=red] (0,-5) -- (1,-5) ;
\node[cross] at (0,-5){}; 
\fill[vertex] (1,-5) circle(2pt);
       \node at (1,-5) [below, color=blue] {$x^2$};
\draw[->, thick,color=green] (0,-5) -- (7,-5)node[right] { \color{red}{$a_1$} \color{green}{$\mathbf{\mathcal{G}_1}$}} ;
 \node at (3,-5) [below, color=blue] {$x^3$};
 \fill[vertex] (3,-5) circle(2pt);
\draw[->, color=red] (3,-5) -- (7,-4)node[right] {$a_2$} ;
\draw[->, color=red] (1,-5) -- (4,-3)node[] {} ;
\fill[vertex] (4,-3) circle(2pt);
\draw[->, color=red] (4,-3) -- (7,-3)node[right] {$a_3$} ;

\draw[->, color=red] (4,-3) -- (5,-2)node[] {} ;
 \node at (4,-3) [below, color=blue] {$x^5$};
 \node at (5,-2) [above, color=blue] {$x^6$};
\fill[vertex] (5,-2) circle(2pt);
\draw[->, color=red] (5,-2) -- (7,-2)node[right] {$a_4$} ;
\draw[->, color=red] (5,-2) -- (7,-1)node[right] {$a_5$} ;

\node at (0.5,-4.8) [below, color=red] {$0$};
\node at (2.3,-4.8) [below, color=red] {$0$};
\node at (4.3,-4.8) [below, color=red] {$0$};
\node at (4.3,-3.8) [below, color=red] {$1$};
\node at (2.3,-3.8) [below, color=red] {$1$};
\node at (2.3,-3.8) [below, color=red] {$1$};
\node at (4.5,-1.5) [below, color=red] {$1$};
\node at (6.3,-0.5) [below, color=red] {$1$};
\node at (6.3,-1.7) [below, color=red] {$0$};
\node at (6.1,-2.7) [below, color=red] {$0$};
\end{tikzpicture}
\caption{Our choice of polynomials realising a given end-rooted complete plane binary tree as their contact tree.\label{fig:exCtcTreeRealised}}
\end{figure}
\end{example}

\begin{proof}

Let us prove now Theorem \ref{Th:OriginalConstructionOneVariable}.

\begin{leftbar}
There are several steps. First, given a separable $(m+1)$-snake $\sigma:\{1,2,\ldots,m+1\}\rightarrow \{1,2,\ldots,m+1\}$ such that $\sigma(m)>\sigma(m+1),$ we construct one of its \textbf{binary} decomposition trees, denoted by $K_+.$ By Definition \ref{def:binSepTreeAssoc}, the leaves of $K_+$ are in a bijective relation with $\{\sigma(1),\ldots,\sigma(m+1)\}$. By Lemma \ref{lemma:RealisingCtctTree} we can  choose real polynomials $a_1(x),\ldots,a_{m+1}(x)$ such that (after adding an extra root to $K_+$, thus transforming it into an \textbf{end-rooted} complete plane binary tree) we have $K_+=\mathcal{CT}(a_1(x),\ldots,a_{m+1}(x))$. By the construction of the contact tree, the polynomials are in bijective correspondence with the leaves of the contact tree. In this proof, by abuse of language, if a leaf  corresponds to the label $\sigma(i)$ we will sometimes label it with  $a_i(x)$, for any $i\in\{1,\ldots,m+1\}$. Since $K_+$ is an embedded tree which represents both a contact tree and a binary separating tree of $\sigma,$ this double labelling of the leaves will enable us to identify in the tree the vertices as follows: $a_i\wedge a_j\equiv \sigma(i)\wedge\sigma(j)$. Furthermore, we define the one real variable polynomials $P_x(y), Q_x(y)\in\mathbb{R}[x][y]$, $P_x(y):=\prod_{i=1}^{m+1}\big(y-a_i(x)\big )$ and then $Q_x(y):=(m+2)\int_{0}^y P_x(t)\mathrm{d}t.$ Denote by $c_i(x)$ the critical values of $Q_x(y)$. Finally, we prove that for a sufficiently small $x>0$, this construction gives us the desired equivalence $c_i(x)>c_j(x)$ if and only if $\sigma(i)>\sigma(j).$
\end{leftbar}

$\bullet$ \textbf{Step $1$:} By hypothesis, $\sigma$ is separable, thus by Proposition  \ref{Prop:charactOfSeparable} $\sigma$ has at least one binary decomposition.  By Proposition \ref{prop:BijectionDecomposAndTress}, the binary decomposition corresponds to a binary separating tree. Let us denote this tree by $K_+$. Now, we obtained $K_+$, a complete plane binary tree. In addition, since $\sigma$ is also an $(m+1)$-snake $\sigma:\{1,2,\ldots,m+1\}\rightarrow \{1,2,\ldots,m+1\}$ such that $\sigma(m)>\sigma(m+1),$ by Lemma \ref{lemma:AlternDecompos} and by Corollary \ref{cor:LastSignMinus} we know that the signs of the internal vertices of $K_+$ alternate and that the rightmost internal vertex is decorated with the $\ominus$ sign (see Example \ref{ex:constr}).

$\bullet$ \textbf{Step $2$:} After adding an extra root to $K_+$, thus transforming it into an \textbf{end-rooted} complete plane binary tree, let us now apply Lemma \ref{lemma:RealisingCtctTree} and construct a set of polynomials $a_i(x)\in \mathbb{R}[x]$ that realise this tree as their contact tree, namely such that $K_+=\mathcal{CT}(a_1(x),\ldots,a_{m+1}(x)).$

$\bullet$ \textbf{Step $3$:} Define the unitary polynomial $P_x(y)\in\mathbb{R}[x][y]$, of degree $m+1$ in the variable $y$ such that the polynomials $a_i(x)$ constructed above are the simple real roots of $P_x(y)$:
$$P_x(y):=\prod_{i=1}^{m+1}\big(y-a_i(x)\big ).$$
Denote by $$\mathrm{S}_i(x):=\bigg \vert \int_{a_i(x)}^{a_{i+1}(x)}P_x(t)\mathrm{d}t\bigg \vert.$$ By Corollary \ref{cor:Diferenta}, we have that the area $\mathrm{S}_m$ is situated below the $Ox$-axis and then the positions of the areas alternate above and below. In other words, we say that an area $\mathrm{S}_i$ is situated below (respectively above) the $Ox$-axis when the integral $\int_{a_i(x)}^{a_{i+1}(x)}P_x(t)\mathrm{d}t$ is negative (respectively positive), thus we associate the minus (respectively plus) sign to the area $\mathrm{S}_i$. 

Similarly, the reader should remember that the signs of a binary decomposition of the separable snake $\sigma$ are alternating and ending with $\ominus.$
Therefore we obtain $$\mathrm{sign}(\sigma(i)\wedge\sigma(i+1))=\ominus\Leftrightarrow \int_{a_i(x)}^{a_{i+1}(x)}P_x(y)\mathrm{d}y<0$$ and  $$\mathrm{sign}(\sigma(i)\wedge\sigma(i+1))=\oplus\Leftrightarrow \int_{a_i(x)}^{a_{i+1}(x)}P_x(y)\mathrm{d}y>0.$$ 

$\bullet$ \textbf{Step $4$:} Furthermore, denote by 
$$Q_x(y):=(m+2)\int_{0}^y P_x(t)\mathrm{d}t,$$
the unitary polynomial in $\mathbb{R}[x][y]$ such that $\frac{\partial Q_x(y)}{\partial y}=(m+2)P_x(y)$, $Q_x(0)=0$. Therefore, the critical points of $Q_x(y)$ are the roots of $P_x(y),$ i.e. the polynomials $a_i(x),$ for $i=1,\ldots,m+1.$ The critical values of $Q_x(y)$ are thus $$c_i(x):=Q_x(a_i(x)),$$ for $i=1,\ldots,m+1.$

$\bullet$ \textbf{Step $5$:} Now by using Corollary \ref{cor:Diferenta}, let us compute the difference between two arbitrary critical values of $Q_x(y),$ say $$c_j(x)-c_i(x)=Q_x(a_{j}(x))-Q_x(a_i(x))=(m+2)\int_{a_i(x)}^{a_{j}(x)}P_x(t)\mathrm{d}t=(m+2)((-1)^{m+1-i}\mathrm{S}_i+\ldots+
(-1)^{m+1-j}\mathrm{S}_{j}).$$ For a better understanding, the reader is invited to see Figure \ref{fig:ValCritCi} below.

\begin{figure}[H]
\includegraphics[scale=0.15]{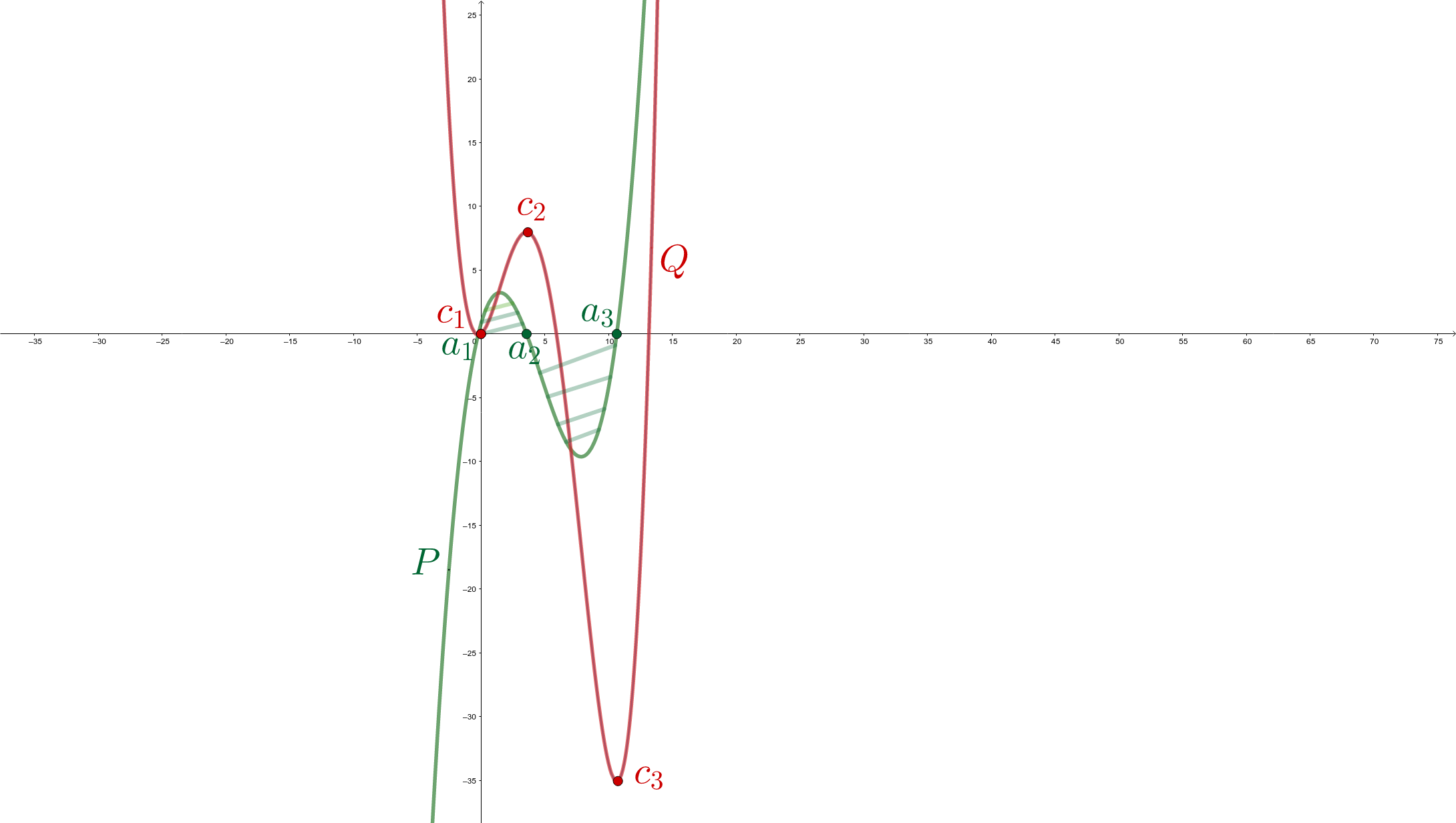} 
\caption{The critical points of $Q_x(y)
$ are the roots of $P_x(y)$ .\label{fig:ValCritCi}}
\end{figure}

By Proposition \ref{prop:valuationOfSumOfAreas}, if $i<j$, we have: 
$$\nu_x\left (\pm\mathrm{S}_i\pm\mathrm{S}_{i+1}
\pm\ldots\pm\mathrm{S}_{j-1}\right )=\nu_x(\mathrm{S}_k),$$ where $\mathrm{S}_k$ denotes the unique area corresponding to the bifurcation vertex $a_i\wedge a_j.$ Thus for small enough $0<x\ll 1$, we have: $c_j(x)-c_i(x)>0$ if and only if the area $\mathrm{S}_k$ is situated above the $Ox$-axis, that is if and only if $\mathrm{sign}(\sigma(i)\wedge\sigma(j))=\oplus$. By Proposition \ref{prop:LegatPermutareSemn}, the last equality is equivalent to $\sigma(j)>\sigma(i)$. 

$\bullet$ \textbf{Step $6$:} The proof is completed: we constructed a polynomial $Q_x(y)\in\mathbb{R}[x][y]$ such that for a sufficiently small $0<x\ll 1$ the critical values of $Q_x(y)$ verify $c_j(x)>c_i(x) \Leftrightarrow \sigma(j)>\sigma(i),$ where $\sigma$ is the given separable snake.  
\end{proof}

\begin{remark*}
Note that the separability of the snakes is due to the hypothesis we impose on the contact trees: they are complete and binary. 
\end{remark*}

\begin{example}\label{ex:constr}
Given the separable snake $\sigma:=\begin{pmatrix}
    1 &2& 3& 4&5 \\
   4&5&1&3&2
  \end{pmatrix}$, let us construct a real univariate $6$-Morse polynomial $Q_x(y)\in\mathbb{R}[x][y]$ such that its associated snake is $\sigma.$
  
\textbf{Step $1$:} obtain a binary separating tree of $\sigma$, denoted by $K_+$, and its associated binary decomposition (see Figure \ref{fig:exBinarySepTreedecomposition} below).

  \begin{figure}[H]
  \centering
\includegraphics[scale=0.5]{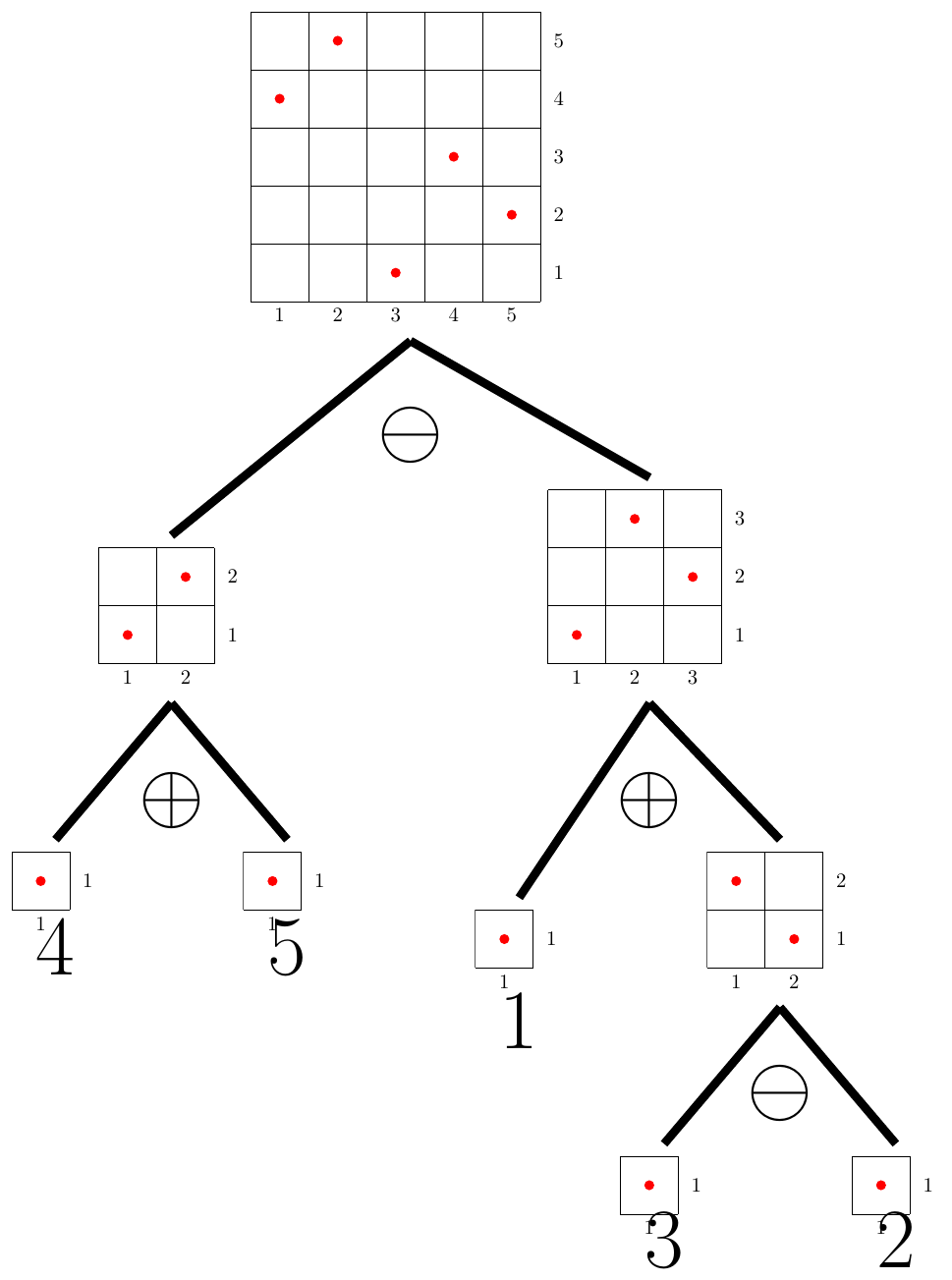} 
  \caption{A binary separating tree of $\sigma=(\boxdot \oplus \boxdot)\ominus (\boxdot\oplus(\boxdot\ominus \boxdot))$. \label{fig:exBinarySepTreedecomposition}}
  \end{figure}

\textbf{Step $2$:} construct the polynomials $a_i(x)$, for $i=1,\ldots,5$ such that $\mathcal{CT}(a_1(x),\ldots,a_5(x))$ is the separating tree $K_+$ of $\sigma$. This has already been done for this tree, in Example \ref{ex:ConstructionOfAi}: $a_1(x)=0,$ $a_2(x)=x^3,$ $a_3(x)=x^2,$ $a_4(x)=x^2+x^5$ and $a_5(x)=x^2+x^5+x^6.$ For a better vision we just rotate the embedded separating tree $\pi/2$ counterclockwise.
  
\textbf{Steps $3$ and $4$:}  For a sufficiently small $0<x\ll 1$, define $P_x(y)\in\mathbb{R}[x][y],$ $P_x(y):=\prod_{i=1}^5(y-a_i(x))$ (see Figure \ref{fig:p}), then $Q_x(y):=6\int_{0}^y P_x(t)\mathrm{d}t$ (see Figure \ref{fig:q}). The critical points of $Q_x(y)$ are the roots of $P_x(y).$ 

\textbf{Step $5$:} Let us denote by $c_i(x):=Q_x(a_i(x))$ the $i$-th critical value of $Q_x(y).$ We have $c_j(x)>c_i(x) \Leftrightarrow \sigma(j)>\sigma(i).$

  \begin{figure}[H]
  \centering
    \includegraphics[scale=0.7]{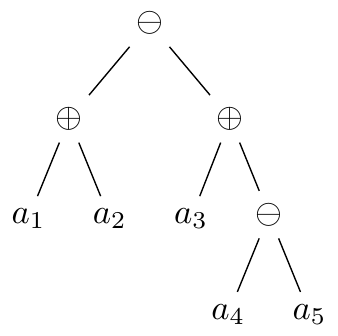} 
    
  \includegraphics[scale=0.3]{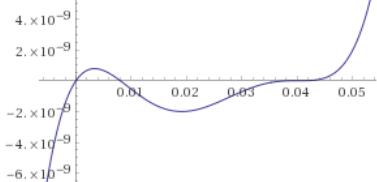} 
  \caption{The graph of $P_x(y)$ for $0<x\ll 1$ small enough: the $5$ roots of the polynomial $P_x(y)$ are in bijective correspondance with the leaves of the binary decomposition tree of the given separable snake $\sigma$; any sign decorating an internal vertex of the tree corresponds with the position of the respective area $\mathrm{S}_i$ ($\oplus$ for an area above $Ox$, respectively $\ominus$ for an area below $Ox$).\label{fig:p}}
  \end{figure}
   \begin{figure}[H]
  \centering
  \includegraphics[scale=0.3]{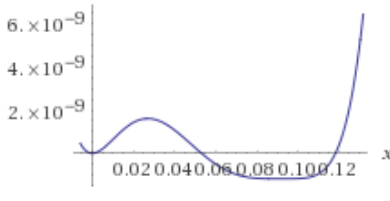} 
  \caption{The graph of $Q_x(y)$ for $0<x\ll 1$ small enough.\label{fig:q}}
  \end{figure}
  
  \begin{figure}[H]
  \centering
  \includegraphics[scale=0.3]{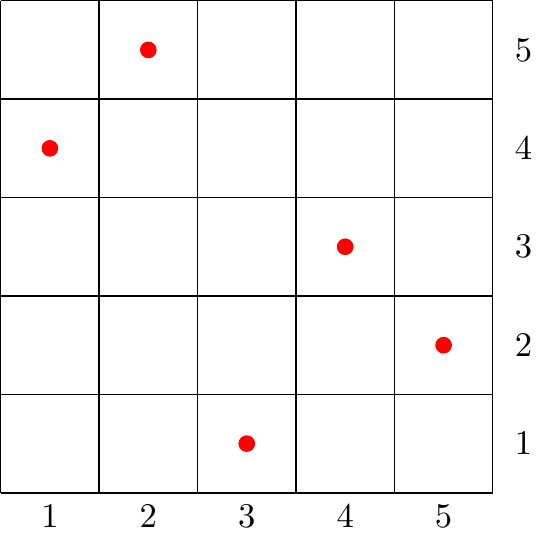} 
  \caption{The Arnold snake $\sigma$ associated to $Q_x(y)$.}
  \end{figure}
\end{example}

\printbibliography
\end{document}